\documentclass[11pt,a4paper]{article}
\usepackage{subfiles}
\usepackage{paper}

\setbool{noProblems}{false}
\setbool{noChanges}{true}
\newcommand{\dobib}{\bibliographystyle{plain}\bibliography{Diss.bib}}

\title{Pure Measures, Density Measures and \\ the Dual of $L^{∞}$}
\author{Moritz Schönherr, Friedemann Schuricht \\ \small TU Dresden -
	Fachrichtung Mathematik\\ \small 01062
Dresden, Germany}
\begin{document}
% This nullifys citation in subfiles
\renewcommand{\dobib}{}

\maketitle

\section*{Abstract}
Measures play an important role in the characterisation of various function
spaces. In this paper, the structure of \tdme{}s will be investigated. These
are elements of the dual of the space of essentially bounded functions. The
main results presented here are a more precise representation of
$\dual{\libld{\dom}}$, leading to the notion of \tpfa{} \tme{}s, and the
definition and analysis of \tdme{}s which constitute a large class of such
\tme{}s. It is shown that \tdme{}s have applications in the context of traces. In particular, new and meaningful examples of \tpfa{} \tme{}s are given
on $\Rn$, in contrast to common examples in the literature, which are usually
constructed on $\N$.

\section{Introduction}

In most mathematical texts, measures are defined to be $\sigma$-additive. For
these \tsme{}s, a rich theory and convergence theorems hold true. Their
importance follows from Riesz Representation Theorem, which characterises the
dual space of the space of all compactly supported functions as the space of
\tRaM{}s. Yet, other dual spaces cannot be represented by \tsme{}s but by
finitely additive \tme{}s, e.g.
$\dual{C_b(\dom)}$ and $\dual{\libld{\dom}}$ (cf.
\cite{alexandroff_additive_1941}, \cite{dunford_linear_1988}). The reason why \tme{}s which are not $\sigma$-additive are not widely used
outside of economic theory is probably due to the lack of meaningful examples
in the literature. In this paper, the basic theory of finitely additive \tme{}s
as used in e.g. \cite{rao_theory_1983} and \cite{dunford_linear_1988} will be
outlined and a new, meaningful example for such \tme{}s will be given on $\Rn$.
Interestingly, many results on their structure are known. In particular, the
dual space of the space of essentially bounded functions is known to be
represented by the space of all bounded measures, which do not charge Lebesgue
null sets. This representation will be
extended by the identification and definition of \textit{\tpfa{}} \tme{}s and
later on \textit{\tdme{}s}, which constitute a large class of \tpfa{} \tme{}s.
Some examples will show that they can be employed in the study of traces and
even differential calculus.

The structure of the paper is as follows. First, some necessary notions and
results from lattice theory will be recalled and useful proposition on the
successive decomposition of vector lattices will be proved.

In the second section, the theory of \tme{}s and the associated integration
theory will be given. The known representation of the dual space of $\libld{\dom}$
will be refined, using the decomposition techniques from the previous section.
Pure \tme{}s will be defined and a necessary condition for a \tme{} to be pure
will be given. In plus, a first example will illustrate the new results.

The last section contains work on \tdme{}s. Following their
definition, existence and some properties will be proved. In plus, the extremal
points of the set of all \tdme{}s will be characterised.
Exemplary applications to the traces of functions of bounded variation and
differential calculus will be presented. Finally, the relation of \tpfa{} \tme{}s
and \tsme{}s which are singular with respect to Lebesgue \tme{} will be
analysed.

Concerning notation, in the following $n\in \N$ denotes a positive natural
number and $\Rn$ the vector space of real $n$-tuples. For a set $\dom \subset
\Rn$ the set $\dnhd{\dom}{\delta}$ denotes the open $\delta$-neighbourhood of
$\dom$. Open balls with radius $\delta >0$ and centre $\eR \in \Rn$ are written
$\ball{\eR}{\delta} = \dnhd{\{\eR\}}{\delta}$. The Borel subsets of $\dom$, i.e. the \tsme{} generated by all
relatively open sets in $\dom$, is denoted by $\bor{\dom}$. $\lem$ is the Lebesgue measure and $\ham^d$ the $d$-dimensional
Hausdorff measure. For set function $\me$ on $\dom$, $\reme{\me}{\als}$ denotes
the restriction of $\me$ to $\als$. The Banach space of equivalence classes of
$p$-integrable functions is denoted by $\lpbld{p}{\dom}$ and $\hocon{p}$
denotes the Hölder-conjugate of $p$. Spaces of continuous functions with a
support which is relatively compact in $\dom$ will be written $\Ccfun{\dom}$. (Weak) Derivates of functions $\fun$ are written
$\Deriv{\fun}$. The divergence of a vector field $\funv$, be it
classical or distributional, is denoted by $\divv{\funv}$.

\section{Tools from Lattice Theory}
First, the basic definitions for vector lattices from Rao \cite[p.
24ff]{rao_theory_1983} (cf. \cite[p. 347]{birkhoff_lattice_1993}) is given.
\begin{definition}
  Let $L$ be a vector space and $\leq$ a partial order on $L$ which is
  compatible with $+$ and the multiplication with a scalar on $L$. If
  for all $l_1,l_2 \in L$ the supremum and infimum of $\{l_1,l_2\}$
  exist, then $L$ is called a \textbf{vector lattice}. For $l,l_1,l_2 \in L$
  write
  \begin{align*}
    l_1 \vee l_2 & := \sup \{l_1,l_2\} \nomenclature[l]{$l_1 \vee
    l_2$}{$\sup\{l_1,l_2\}$} \\
    l_1 \wedge l_2 & := \inf \{l_1,l_2\} \nomenclature[l]{$l_1 \wedge
    l_2$}{$\inf\{l_1,l_2\}$}\\
    l^+ &:= l \vee 0 \nomenclature[l]{$l^+$}{$l\vee 0$}\\
    l^- & := -l \vee 0 \nomenclature[l]{$l^-$}{$l \wedge 0$} \\
    \tova{l} & := l^+ + l^- \nomenclature[l]{$\tova{l}$}{$l^+ + l^-$}
  \end{align*}
  $l_1,l_2 \in L$ are called \textbf{orthogonal}\index{orthogonal lattice
  elements}, if $|l_1|\wedge
  |l_2| = 0$, written $l_1 \perp l_2$.\nomenclature[l]{$l_1 \perp l_2$}{$\tova{l_1}
	  \wedge \tova{l_2} = 0$}  
  If for a family $\{l_i\}_{i\in \mathcal{I}} \subset L$ the supremum
  exists, write
  \begin{equation*}
    \bigvee \limits_{i \in \mathcal{I}} l_i := \sup \limits_{i \in
      \mathcal{I}} l_i \text{\,.}
  \end{equation*}
  If the infimum of $\{l_i\}_{i \in \mathcal{I}}$ exists, it is
  denoted by
  \begin{equation*}
    \bigwedge \limits_{i \in \mathcal{I}} l_i := \inf \limits_{i \in
      \mathcal{I}} l_i \text{\,.}
  \end{equation*}
  A set $L' \subset L$ is called \textbf{bounded from above}, if there exists $l
  \in L$, such that $l' \leq l$ for all $l' \in L'$.

  A vector lattice is called \textbf{boundedly complete}\index{boundedly
  complete lattice}, if for every
  $\{l_i\}_{i\in \mathcal{I}} \subset L$ which is bounded from above
  the supremum $\bigvee \limits_{i\in \mathcal{I}} l_i$ exists.
\end{definition}
For a vector lattice $L$ and $l_1,l_2 \in L$
\begin{equation*}
  \tova{l_1 + l_2} \leq \tova{l_1} + \tova{l_2}
\end{equation*}
with equality if $l_1 \perp l_2$ (cf. \cite[p. 25]{rao_theory_1983}).
The following example foreshadows the partial order that turns spaces
of \tme{}s into vector lattices.

In order to obtain results for an orthogonal decomposition of vector
lattices (and their elements), one has to define appropriate sub-structures (cf. \cite[p. 28]{rao_theory_1983}).
\begin{definition}
  A linear subspace $L'$ of $L$ is called a \textbf{sublattice} of $L$
  if $l_1 \vee l_2 \in L'$ and $l_1 \wedge l_2 \in L'$ for all $l_1,l_2$ in $L'$.
  
  A sublattice $L'$ of $L$ is called \textbf{normal}\index{normal sublattice}, if
  \begin{enumerate}
  \item for all $l' \in L'$ and all $l\in L$ 
    \begin{equation*}
      |l| \leq |l'| \implies l \in L'
    \end{equation*}
  \item if for $\{l_i\}_{i\in \mathcal{I}} \subset L'$ the supremum
    exists in $L$, then $\bigvee\limits_{i\in \mathcal{I}} l_i \in L'$.
  \end{enumerate}
\end{definition}
In order to decompose a vector lattice into normal sublattices, a notion of
orthogonality is needed (cf. \cite[p. 29]{rao_theory_1983}).
\begin{definition}\label{def:orth_comp}
  For a subset $L'$ of $L$, the set
  \begin{equation*}
    (L')^\perp := \{l \in L \setpipe \forall l' \in L': l \perp l'\} 
    \nomenclature[l]{$L^\perp$}{orthogonal complement}
  \end{equation*}
  is called \textbf{orthogonal complement} of $L'$.
\end{definition}
The following statement from \cite[p. 29f]{rao_theory_1983} illustrates that
normal sublattices and orthogonality interact in a similar way as closed linear
subspaces and orthogonality in Hilbert spaces do.
\begin{proposition}
  Let $S \subset L$, then $S^\perp$ is a normal
  sublattice of $L$. If $S$ is a normal sublattice, then $(S^\perp)^\perp = S$.
\end{proposition}
A useful characterisation of the orthogonal complement of a normal sublattice
is the following.
\begin{proposition}\label{prop:char_orth}
	Let $S$ be a normal sublattice of $L$. Then $l \in S^\perp$ if and only
	if for every $s \in S$
	\begin{equation*}
		0 \leq \tova{s} \leq \tova{l} \implies s = 0 \,.
	\end{equation*}
\end{proposition}
\begin{proof}
	Assume first that $l \in S^\perp$. Then for every $s \in S$
	\begin{equation*}
		0 \leq \tova{s} \leq \tova{l} \implies 0 = \tova{s} \wedge
		\tova{l} = \tova{s} \implies s = 0 \,.
	\end{equation*}
	Now assume for every $s \in S$
	\begin{equation*}
		0 \leq \tova{s} \leq \tova{l} \implies s = 0 \,.
	\end{equation*}
	Since $S$ is a normal sublattice
	\begin{equation*}
		0 \leq \tova{s} \wedge \tova{l} \leq \tova{s} \implies
		\tova{s} \wedge \tova{l} \in S \,.	
	\end{equation*}
	By assumption
	\begin{equation*}
		\tova{s} \wedge \tova{l} \leq \tova{l}\implies  \tova{s}\wedge
		\tova{l} = 0 \,.
	\end{equation*}
	Thus $s \perp l$.
\end{proof}
As in the setting of Hilbert spaces, a boundedly complete vector lattice can be
represented as the direct sum of a normal sublattice and its orthogonal
complement (cf. \cite[p. 29]{rao_theory_1983}).
\begin{proposition}\label{thm:riesz_dec}{Riesz Decomposition Theorem\\}
	Let $S$ be a normal sublattice of $L$, then for every $l \in L$ there
  exist unique elements $s\in S,s^\perp \in S^\perp$ such that
  \begin{equation*}
    l = s + s^\perp \text{\,.}
  \end{equation*}
  Furthermore, if $l \geq 0$, then $s = \bigvee \limits_{s' \in S} l
  \wedge |s'|$. For general $l \in L$ 
  \begin{equation*}
    s = \bigvee \limits_{s' \in S} l^+ \wedge |s'| -
    \bigvee\limits_{s'\in S} l^- \wedge |s'| \text{\,.}
  \end{equation*}
\end{proposition}

The following proposition enables the successive decomposition of a lattice
into sublattices. This is used in the analysis of \tme{}s. In particular,
this proposition enables a better characterisation of the dual of the space of
essentially bounded functions.
\begin{proposition}\label{thm:lat_capstab}
  Let $L_1,L_2$ be two normal sublattices of $L$. Then $L_1 \cap L_2$
  is a normal sublattice of $L_2$. Furthermore, the orthogonal
  complement of $L_1 \cap L_2$ in $L_2$ is $L_1^\perp \cap L_2$.
\end{proposition}
\begin{proof}
	Let $l_1\in L_1 \cap L_2$ and $l_2 \in L_2$ with
	\begin{equation*}
		\tova{l_2} \leq \tova{l_1} \,.
	\end{equation*}
	Since $L_1$ is a normal sublattice of $L$, 
	\begin{equation*}
		l_2 \in L_1 \,.
	\end{equation*}
	Whence $l_2 \in L_1 \cap L_2$.

	Now, let $\{l_i\}_{i \in\mathcal{I}}\subset L_1 \cap L_2$ be such that
	$\bigvee\limits_{i \in \mathcal{I}} l_i \in L$. Since $L_1$
	and $L_2$ are normal, 
	\begin{equation*}
		\bigvee\limits_{i \in \mathcal{I}} l_i \in L_1 \quad \text{ and
		} \quad \bigvee\limits_{i \in \mathcal{I}} l_i \in L_2 \,.
	\end{equation*}
	This implies $\bigvee \limits_{i \in \mathcal{I}}l_i \in L_1 \cap L_2$.
	Thus $L_1 \cap L_2$ is a normal sublattice of $L_2$.

  Let $l_2 \in L_2$ such that $l_2 \in
  (L_1\cap L_2)^\perp$. Since $L_1$ is a normal sublattice of $L$, there exist $l_1 \in L_1,
  l_1^\perp \in L_1^\perp$ such that $l_2 = l_1 + l_1^\perp$. Now, using
  additivity of the total variation on orthogonal elements (cf.
  \cite[p. 25]{rao_theory_1983})
  \begin{equation*}
    0 \leq \sup \{|l_1|,|l_1^\perp|\} \leq |l_1| + |l_1^\perp| = |l_2| \,.
  \end{equation*}
  Hence,  $l_1,l_1^\perp \in L_2$ and $l_1,l_1^\perp \in (L_1 \cap L_2)^\perp$. Since $l_2 \in (L_1 \cap
  L_2)^\perp$, 
  \begin{equation*}
    0 = |l_2| \wedge |l_1| = |l_1 | \wedge |l_1| + |l_1| \wedge
    |l_1^\perp | = |l_1| \wedge |l_1| \,.
  \end{equation*}
  This implies $l_1 = 0$. Hence
  \begin{equation*}
    (L_1 \cap L_2)^\perp  \subset L_1^\perp \cap L_2 \text{\,.}
  \end{equation*}
  On the other hand, if $l_1^\perp \in L_1^\perp \cap L_2$, then for
  all $l_1 \in L_1 \cap L_2$
  \begin{equation*}
    |l_1| \wedge |l_1^\perp| = 0\,,
  \end{equation*}
  whence
  \begin{equation*}
    L_1^\perp \cap L_2 \subset (L_1 \cap L_2)^\perp \text{\,.}
  \end{equation*}
\end{proof}

\section{A Primer On Pure Measures}
In this article, set functions $\me : \al \subset \pos{\dom}\to \R$ will be called
\textit{\tme{}}, if for all $m \in \N$ and every pairwise disjoint $\{\als_k\}_{k=1}^m\subset \al$ with
$\bigcup \limits_{k=1}^m \als_k \in \al$
\begin{equation*}
	\me\left(\bigcup \als_k\right) = \sum \limits_{k=1}^m \me(\als_k) \,.
\end{equation*}
If this holds with $m= \infty$, the \tme{} is called \tsme{}. A \tme{} is
called \textit{bounded} if 
\begin{equation*}
	\sup \limits_{\als \in \al} |\me(\als)| < \infty \,.
\end{equation*}
An \textit{\tal{}} is a class of sets which is stable under union, intersection and
differences and contains at least $\emptyset$.
The spaces of \tme{}s considered in this paper are defined in accordance with \cite{rao_theory_1983}.

The spaces of \tme{}s considered in this thesis are defined in accordance with \cite{rao_theory_1983}.
\begin{definition}
	Let $\dom \subset \Rn$ and $\al \subset \pos{\dom}$ be an algebra. The set of all bounded \tme s $\me : \al \to \R$ is denoted by
	\begin{equation*}
		\baa  \,.
		\nomenclature[m]{$\baa$}{space of bounded \tme{}s on $\dom$}
	\end{equation*}
	The set of all bounded \tsme{}s $\sme : \al \to \R$ is denoted by 
	\begin{equation*}
		\caa \,.		
		\nomenclature[m]{$\caa$}{space of bounded \tsme{}s on $\dom$}
	\end{equation*}
\end{definition}
There is a natural partial order on $\baa$ (cf. \cite[p. 43]{rao_theory_1983}).
\begin{definition}
	Let $\dom \subset \Rn$ and $\al \subset \pos{\dom}$ be an \tal{}.
	For $\me,\metoo\in \baa$ one writes 
	\begin{equation*}
		\me \leq \metoo
	\end{equation*}
	if and only if for every $\als \in \al$
	\begin{equation*}
		\me(\als) \leq \metoo(\als) \,.
		\nomenclature[m]{$\me \leq \metoo$}{$\me(\als) \leq
		\metoo(\als)$ for all $\als \in \al$}
	\end{equation*}
\end{definition}
The following proposition links the theory of \tme{}s with the theory of
boundedly complete vector lattices. This is essential for the subsequent
results on the decomposition of \tme{}s.
The proposition is taken from \cite[p. 43f]{rao_theory_1983}. 
\begin{proposition}\label{prop:ba_bcsubl}
	Let $\dom \subset \Rn$ and $\al \subset \pos{\dom}$ be an $\tal{}$.
	Then $\baa$ together with the partial order $\leq$ is a boundedly
	complete vector lattice.
\end{proposition}
The following definitions are standard in \tme{} theory (cf. \cite[p. 45]{rao_theory_1983}).
\begin{definition}
	Let $\dom \subset \Rn$ and $\al \subset \pos{\dom}$ be an \tal{}. For $\me \in \baa$ define
	\begin{align*}
		\posi{\me} & := \me \vee 0 = \sup\{\me,0\}\\
		\nega{\me} & := (-\me) \vee 0 = \sup\{-\me,0\}\\
		\tova{\me} & := \posi{\me} + \nega{\me} \,.
	\end{align*}
	Call $\posi{\me}$ \textbf{\tposi}\index{positive@\tposi{} of a
	\tme{}}\nomenclature[m]{$\posi{\me}$}{\tposi{} of $\me$} of $\me$,
	$\nega{\me}$ \textbf{\tnega}\index{negative@\tnega{} of a \tme{}}
	\nomenclature[m]{$\nega{\me}$}{\tnega{} of $\me$} of $\me$ and
	$\tova{\me}$ \textbf{\ttova}\index{total@\ttova{} of a
	\tme{}}\nomenclature[m]{$\tova{\me}$}{\ttova{} of $\me$} of $\me$.

	Furthermore, for $\als \in \al$ define $\reme{\me}{\als} : \al \to \R$
	by
	\begin{equation*}
		(\reme{\me}{\als})(\als') := \me(\als \cap \als') \text{ for all } \als' \in \al \,.
		\nomenclature[m]{$\reme{\me}{\als}$}{restriction of $\me$ to
	$\als$} 
	\end{equation*}
\end{definition}
The \ttova{} and the lattice operations can be characterised in the following
way (cf. \cite[p. 46]{rao_theory_1983}, \cite[p. 48]{yosida_finitely_1951}).
\begin{proposition}
	Let $\dom \subset \Rn$ and $\al \subset \pos{\dom}$ be an \tal{}. Then
	for every $\me, \metoo \in \baa$ and $\als \in \al$
	\begin{equation*}
		\tova{\me}(\als) = \sup \sum \limits_{k=1}^m \left|\me\left(\als_k\right)\right|
	\end{equation*}
	and
	\begin{equation*}
		(\me \wedge \metoo)(\als) = \inf \limits_{\substack{\als' \in
		\al, \\ \als' \subset \als}} \me(\als')  + \metoo(\als
			\setminus \als') \,.
	\end{equation*}
	where the supremum is taken over all finite partitions $\{\als_k\}_{k = \mN}^m \subset \al$ of $\als$.
\end{proposition}

The following proposition can be found in Rao \cite[p. 44]{rao_theory_1983}. It
states that in the space of bounded \tme{}s, the norm is compatible with the
partial order.
\begin{proposition}
	Let $\dom \subset \Rn$ and $\al \subset \pos{\dom}$ be an \tal{}. Then
	$\baa$ together with $\leq$ and the norm
	\begin{equation*}
		\norm{}{\me} :=\tova{\me}(\dom) \text{ for } \me \in \baa
		\nomenclature[n]{$\norm{}{\me}$}{$\tova{\me}(\dom)$}
	\end{equation*}
	is a Banach lattice, i.e. it is a Banach space and a vector lattice such that for all $\me,\metoo \in
	\baa$
	\begin{equation*}
		\tova{\me} \leq \tova{\metoo} \implies \norm{}{\me} \leq
		\norm{}{\metoo} \,.
	\end{equation*}
\end{proposition}
The following proposition is an application of Riesz's decomposition Theorem
(Proposition \ref{thm:riesz_dec}) (cf. \cite[p. 241]{rao_theory_1983}). In
particular, every bounded \tme{} can be uniquely decomposed into a \tsme{} and
a \tpfa{} \tme{}. Recall
the definition of orthogonal complement from page \pageref{def:orth_comp}.
\begin{proposition}\label{prop:yosh_hew_dec}
	Let $\dom \subset \Rn$ and $\al \subset \pos{\dom}$ be an algebra. Then $\baa$ is a boundedly complete vector lattice and $\caa$ one of its normal sublattices. Hence, every $\me \in \baa$ can uniquely be decomposed into $\me_c \in \caa$ and $\me_p \in \caa^\perp$ such that
	\begin{equation*}
		\me = \me_c + \me_p
	\end{equation*}
	\nomenclature[m]{$\me_c$}{$\sigma$-additive part of
	$\me$}\nomenclature[m]{$\me_p$}{\tpfa{} part of $\me$}
	and for every $\sme \in \caa$
	\begin{equation*}
		0 \leq \sme \leq |\me_p| \implies \sme = 0 \,.
	\end{equation*}
\end{proposition}
\begin{definition}
	Let $\dom \subset \Rn$ and $\al \subset \pos{\dom}$ be an \tal{}. Then
	every \tme{} $\me_p \in \caa^\perp$ is called
	\textbf{\tpfa}\index{pure@\tpfa{} \tme{}}.
	Notice that $\me_p$ is not $\sigma$-additive, by definition.
\end{definition}
One important example of \tme{}s that are \tpfa{} are density
\tme{}s. The following new example presents a particular \tdme{}, namely a
density at zero. In the literature, examples of \tpfa{} \tme{} are only known
for $\dom = \N$ (cf. \cite[p. 247]{rao_theory_1983}), they are defined on very
small \tal{}s (cf. \cite[p. 246]{rao_theory_1983}) or they are constructed in such a way that the \tme{} cannot be
computed explicitly, even on simple sets (cf. \cite[p.
57f]{yosida_finitely_1951}). The example given here is constructed on $\dom =
\Rn$ and lives on the Borel subsets of $\dom$.

\begin{example}\label{ex:dzero}
	Let $\dom := \ball{0}{1}\subset \Rn$ be open. Then there exists $\me
	\in \baA{\bor{\dom}}$, $\me \geq 0$ such that
	for every $\bals \in \bor{\dom}$
	\begin{equation*}
		\me(\bals) = \lim \limits_{\delta \downarrow 0} \frac{\lem(
		\bals \cap \ball{0}{\delta})}{\lem(\ball{0}{\delta})}
	\end{equation*}
	if this limit exists. This \tme{} is non-unique. Its existence is 
	shown in Proposition \ref{prop:ex_densme} (take $\lambda:= \lem$ and
	$\ferm = \{0\}$). 
		
	It is shown in Example \ref{ex:dzero_pfa}
	that $\me$ is indeed \tpfa{}. Figure \ref{fig:dzero_pfa} shows the
	family $\{\als_k\}_{k\in \N} \subset \bor{\dom}$
	\begin{equation*}
		\als_k := \left[\frac{1}{k+2},\frac{1}{k+1}\right) \times
		[-1,1]^{n-1} \,.
	\end{equation*}
	For this family
	\begin{equation*}
		\sum \limits_{k\in \N} \me(\als_k \cap \dom) = 0 \neq
		\me\left(\left(0,\frac{1}{2}\right)\times [-1,1]^{n-1}\cap \dom \right) = \me
		\left(\bigcup\limits_{k =1}^\infty \als_k \cap \dom\right)\,.
	\end{equation*}
	Hence, $\me$ is not a \tsme{}.
\end{example}

\begin{figure}[H]
	\centering
	\begin{tikzpicture}[scale=1.2]
		\draw (-0.2,-0.1) node {x};
		\draw[gray,thick,->] (0,0) -- node[above] {$\delta$} (150:2.2);
		\draw[fill] (0,0) circle [radius = 0.01];
		\draw[gray,dashed] (0,0) circle [radius=2.2];
		\draw (0.1,-2) rectangle (0.2,2);
		\draw (0.2,-2) rectangle (0.4,2);
		\draw (0.4,-2) rectangle (0.7,2);
		\draw (0.7,-2) rectangle (1.0,2);
		\draw (1.0,-2) rectangle node {$A_k$}(1.5,2);
		\draw (1.5,-2) rectangle node {...} (2.1,2);
		\draw (2.1,-2) rectangle node {$A_2$} (2.8,2);
		\draw (2.8,-2) rectangle node {$A_1$} (3.8,2);
	\end{tikzpicture}
	\caption{A family of sets on which $\me$ is not
	$\sigma$-additive}\label{fig:dzero_pfa}
\end{figure}
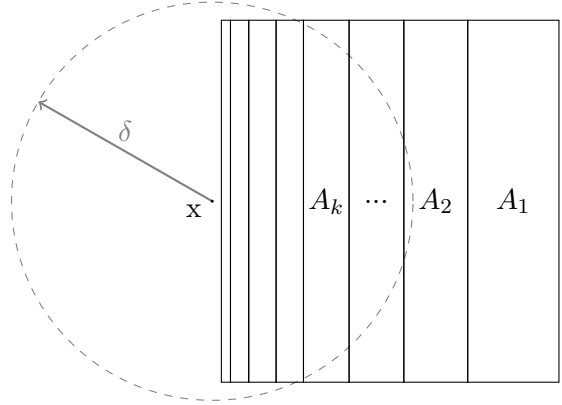

Measures that do not charge sets of Lebesgue measure
zero are of special interest, because these \tme{}s lend themselves naturally to the integration of
functions that are only defined outside of a set of measure zero. When treating non $\sigma$-additive measures, one carefully has to
distinguish the following two notions (cf. \cite[p. 159]{rao_theory_1983}).
\begin{definition}
	Let $\dom \subset \Rn, \al \subset \pos{\dom}$ be an algebra and
	$\metoo\in \baa$. Then $\me \in \baa$ is called 
	\begin{enumerate}
		\item \textbf{\tac}\index{absolutely@\tac{}} with respect to
			$\metoo$, if for every $\eps > 0$ there exists $\delta > 0$ such that for all $\als \in \al$
			\begin{equation*}
				|\metoo|(\als) < \delta \implies |\me(\als)| < \eps \,.
				\nomenclature[m]{$\me \ac \metoo$}{$\me$ is \tac{} w.r.t.
			$\metoo$}
			\end{equation*}
			In this case, write $\me \ac \metoo$.
		\item \textbf{\twac}\index{weakly absolutely@\twac{}} with
			respect to $\metoo$, if for every $\als \in \al$
			\begin{equation*}
				|\metoo|(\als) = 0 \implies \me(\als) = 0 \,.
				\nomenclature[m]{$\me \wac \metoo$}{$\me$ is \twac{}
			w.r.t. $\metoo$}
			\end{equation*}
			In this case, write $\me \wac \metoo$.
	\end{enumerate}
	The set of all \twac{} \tme{}s in $\baa$ is denoted by 
	\begin{equation*}
		\baaw{\metoo} \,.
		\nomenclature[m]{$\baaw{\metoo}$}{bounded and \twac{} measures
		w.r.t. $\metoo$}
	\end{equation*}
\end{definition}
The following proposition shows that there is no \tpfa{} \tme{} which is \tac{}
with respect to some \tsme{} (cf. \cite[p. 163]{rao_theory_1983}).
\begin{proposition}
	Let $\dom \subset \Rn, \al \subset \pos{\dom}$ be an \tal{} and $\sme \in \caa$. Then for every $\me \in \baa$
	\begin{equation*}
		\me \ac \sme \implies \me \in \caa\,.
	\end{equation*}
\end{proposition}
\begin{remark}
	The preceding proposition shows that one should focus on the notion of
	weak absolute continuity when studying \tme{}s that are continuous with respect to some \tsme{}.
\end{remark}
\begin{example}
	$\me$ from Example \ref{ex:dzero} is even \twac{} with respect to
	$\lem$. This is evident from the construction in Proposition
	\ref{prop:ex_densme} (take $\lambda:= \lem$ and $\ferm := \{0\}$).	
\end{example}
\begin{proposition}\label{thm:wac_totvar}
	Let $\me_1, \me_2 \in \baa$ be such that $\me_1 \wac \me_2$. If $\als \in \al$ such that $|\me_2|(\als) = 0$, then $|\me_1|(\als) = 0$.
\end{proposition}
\begin{proof}
	Since $|\me_2|$ is monotone, 
	\begin{equation*}
		|\me_2(\als')| \leq |\me_2|(\als') \leq |\me_2|(\als) = 0
	\end{equation*}
	for all $\als' \in \al$ such that $\als' \subset \als$. Since
	\begin{equation*}
		\me_1^+(\als) = \sup \limits_{\substack{\als' \in \al\\\als' \subset \als}} \me_1(\als') = 0 
	\end{equation*}
	and a similar equation holds for $\me_1^-$
	\begin{equation*}
		|\me_1|(\als) = \me_1^+(\als) + \me_1^-(\als) = 0 \,.
	\end{equation*}
\end{proof}
The following proposition is the key to decompose \tme{}s into \tsme{}s which
are \twac{} with respect to some \tme{} and \tpfa{} \tme{}s. 
\begin{proposition}\label{thm:baw_normsublat}
	Let $\dom \subset \Rn$, $\al \subset \pos{\dom}$ be an algebra and
	$\metoo \in \baa$.

	Then $\baaw{\metoo}$ is a normal sublattice of $\baa$ and thus a boundedly complete vector lattice.
\end{proposition}
\begin{proof}
	$\baaw{\metoo}$ is obviously a linear space. Let $\{\me_i\}_{i\in
	\mathcal{I}} \subset \baaw{\metoo}$ be such that there exists $\me \in \baa$ with 
	\begin{equation*}
		\me_i \leq \me \text{ for all } i \in \mathcal{I} \,.
	\end{equation*}
	By Proposition \ref{prop:ba_bcsubl}, $\baa$ is boundedly complete (cf.
	\cite[p. 44]{rao_theory_1983}). Hence, there exists $\me' \in \baa$ such that 
	\begin{equation*}
		\me_i \leq \me' \text{ for all } i \in \mathcal{I}
	\end{equation*}
	and if this holds true for another $\me'' \in \baa$ then $\me' \leq \me''$.

	Assume $\me' \notin \baaw{\metoo}$. Then there exists $\als \in \al$ such that
	\begin{equation*}
		|\metoo|(\als) = 0 \text{ but } \me'(\als) \neq 0 \,.
	\end{equation*}
	Now, $|\me'\lfloor \als| \in \baa$. Whence $\me' - |\me'\lfloor \als|
	\in \baa$. Since $\me_i(\als) = 0$ 
	\begin{equation*}
		\me_i \leq \me' - |\me' \lfloor \als| < \me' \text{ for all } i \in \mathcal{I} \,,	
	\end{equation*}
	in contradiction to the minimality of $\me'$. Hence $\me' \in
	\baaw{\metoo}$.

	Now let $\me' \in \baa$ and $\me \in \baaw{\metoo}$ such that $|\me'|
	\leq |\me|$. Let $\als \in \al$ be such that $|\metoo|(\als) = 0$. Then 
	\begin{equation*}
		|\me'(\als)| \leq |\me'|(\als) \leq |\me|(\als) = 0
	\end{equation*}
	by Proposition \ref{thm:wac_totvar}. Hence $\me' \in \baaw{\metoo}$.
	Therefore, $\baaw{\metoo}$ is a normal sublattice and thus a boundedly
	complete vector lattice.
\end{proof}
The proposition above enables the decomposition of \tme s into \tpfa{} parts
and \tsme{}s, analogously to Proposition \ref{prop:yosh_hew_dec}.
\begin{theorem}\label{thm:dec_pfa_wac}
	Let $\dom \subset \Rn$, $\al \subset \pos{\dom}$ be an algebra and
	$\metoo \in \baA{\al}$.

	Then for every $\me \in \baaw{\metoo}$ there exist unique $\me_c \in \caa
	\cap \baaw{\metoo}$, $\me_p \in \caa^\perp \cap \baaw{\metoo}$ such that
	\begin{equation*}
		\me = \me_c + \me_p \,.
	\end{equation*}
\end{theorem}
\begin{proof}
	Since $\baaw{\metoo}$ and $\caa$ are normal sublattices of $\baa$, Proposition \ref{thm:lat_capstab} yields that
	\begin{equation*}
		\caa \cap \baaw{\metoo}
	\end{equation*}
	is a normal sublattice of $\baaw{\metoo}$ whose orthogonal complement is 
	\begin{equation*}
		\caa^\perp \cap \baaw{\metoo} \,.
	\end{equation*}
	This, together with Riesz's decomposition Proposition
	\ref{thm:riesz_dec}, yields the statement of the theorem.
\end{proof}
\begin{example}\label{ex:dzero_pfa}
	Since the \tme{} $\me$ from Example \ref{ex:dzero} is positive and
	$\me_c \perp \me_p$, using the additivity of the total variation on
	orthogonal element (cf. \cite[p. 25]{rao_theory_1983}) yields
	\begin{equation*}
		0 \leq \tova{\me_c} \leq \tova{\me_c} + \tova{\me_p} = \tova{\me} = \me \,.
	\end{equation*}
	Hence, for every $\delta > 0$
	\begin{equation*}
		\tova{\me_c}(\ball{0}{\delta}^c)  = 0 \,.
	\end{equation*}
	Thus
	\begin{equation*}
		\tova{\me_c}(\dom \setminus \{0\}) = \lim \limits_{\delta \downarrow 0}
		\tova{\me_c}(\ball{0}{\delta}^c) = 0  \,.
	\end{equation*}
	But $\tova{\me_c}(\{0\}) \leq \me(\{0\})=0$. Hence
	\begin{equation*}
		\tova{\me_c}(\dom) = 0 
	\end{equation*}
	and $\me = \me_p$ is \tpfa{}.
\end{example}
When $\metoo$ is a \tsme{}, the structure of $\me_c$ is well known by the Radon
Nikodym theorem (cf. \cite[p. 128ff]{halmos_measure_1974}).
\begin{proposition}{Radon-Nikodym Theorem\\}\label{thm:radon_nik}
	Let $\dom \subset \Rn$ and $\sal \subset \pos{\dom}$ be a \tsal{}. Furthermore, let $\sme \in \cas$ and $\me \in \cas$ be such that $\me \wac \sme$. Then there exists $\fun \in \lpss{1}$ such that
	\begin{equation*}
		\me(\als) = \I{\als}{\fun}{\sme}
	\end{equation*}
	for every $\als \in \sal$.
\end{proposition}
The structure of $\me_p$ is described by the following proposition taken from
\cite[p. 244]{rao_theory_1983} (cf. \cite[p. 56]{yosida_finitely_1951}).
\begin{remark}
	The following results are stated for \tsme{}s $\sme \geq 0$. They also
	hold for arbitrary \tsme{}s $\sme$ when using $\tova{\sme}$.
\end{remark}
\begin{proposition}\label{prop:pfa_iff_supseq}
	Let $\dom\subset \Rn$, $\sal \subset \pos{\dom}$ be a \tsal{} and $\sme
	\in \cas$, $\sme \geq 0$.  Then $\me \in \basws$ is \tpfa{} if and only if there exists a decreasing sequence $\{A_k\}_{k\in \N}\subset \sal$ such that
	\begin{equation*}
		\sme(\als_k) \xrightarrow{k \to \infty} 0
	\end{equation*}
	and for all $k \in \N$
	\begin{equation*}
		|\me_p|(\als_k^c) = 0 \,.
	\end{equation*}
\end{proposition}
Intuitively speaking, \twac{} \tme{}s are \tpfa{} if and only if they concentrate in the vicinity of a set of \tme{} zero.
Reviewing Example \ref{ex:dzero}, the support (cf. \cite[p.30]{Ambrosio2000}) of the \tme{} can be seen to lie
outside of $\dom\setminus \{0\}$. Yet the construction of the \tme{} would
still work on this set. Hence, it is possible for a \tpfa{} \tme{} to have support outside
of its domain of definition. This necessitates the following definition of
\tcor{}.
\begin{definition}
	Let $\dom \subset \Rn$, $\al \subset \pos{\dom}$ be an \tal{} containing every relatively open set in $\dom$. Furthermore let $\me \in \baa$. Then the set
\begin{equation*}
	\cor{\me} := \{ \eR \in \Rn\setpipe |\me|(\nhd \cap
	\dom) > 0, \forall \nhd \subset \Rn, \nhd \text{ open}, \eR\in \nhd\} 
	\end{equation*}
	\nomenclature[m]{$\cor{\me}$}{\tcor{} of $\me$}
	is called \textbf{\tcor}\index{core@\tcor{}} of $\me$.

	 Let $d\in [0,n]$ be the Hausdorff dimension of $\cor{\me}$. Then $d$
	 is called \textbf{\tcordimo{\me}}\index{core dimension@\tcordimo{\me}} and $\me$ is called \textbf{\tdd{d}}.
\end{definition}
\begin{remark}
	Note that there is a slight difference to the notion of
	\textit{support} of a measure as defined in classic measure theory (cf.
	\cite[p. 60]{federer_geometric_1996}). 
	The \tcor{} of a measure is not necessarily contained in $\dom$, the support of a \tsme{} is.
\end{remark}
\begin{example}
	The measure $\me$ from Example \ref{ex:dzero} has
	\begin{equation*}
		\cor{\me} = \{0\}
	\end{equation*}
	and is thus \tdd{0}.
\end{example}
Now, an example for a \tdme{} with a larger core is given. Note that in this
thesis
\begin{equation*}
	\dnhd{\ferm}{\delta} := \distf{\dom}^{-1}((-\infty,\delta))  \text{
	for } \ferm \subset \Rn \,.
\end{equation*}
\begin{example}
	Let $\dom \subset \Rn$ be open, $d\in [0,n)$ and $\ferm \subset
	\cl{\dom}$ be closed with Hausdorff dimension $d$. Then there
	exists a \tpfa{} \tme{} $\me \in \bawl{\dom}$, $\me \geq 0$ such that for every $\bals \in
	\bor{\dom}$
	\begin{equation*}
		\me(\bals) = \lim \limits_{\delta\downarrow 0} \frac{\lem(\bals
			\cap \dnhd{\ferm}{\delta}\cap \dom)}{\lem(\dnhd{\ferm}{\delta}
		\cap \dom)} =: \denss{\ferm}{\bals} \,,
		\nomenclature[m]{$\denss{\ferm}{\bals}$}{$\lim
			\limits_{\delta\downarrow 0} \frac{\lem(\bals
			\cap \dnhd{\ferm}{\delta}\cap \dom)}{\lem(\dnhd{\ferm}{\delta}\cap \dom)}$} 
	\end{equation*}
	if this limit exists. Here, $\dnhd{\ferm}{\delta}$ is the open 
	$\delta$-neighbourhood of $\ferm$. Furthermore
	\begin{equation*}
		\cor{\me} = \ferm 
	\end{equation*}
	and $\me$ is thus \tdd{d}.

	The existence of this measure is evident by Proposition
	\ref{prop:ex_densme} (take $\lambda:= \lem$).
\end{example}
\begin{proposition}
	Let $\dom \subset \Rn$ and $\al \subset \pos{\dom}$ be an \tal{}
	containing every relatively open set and $\me \in \baa$. Then
	$\cor{\me}$ is a closed set in $\Rn$.
\end{proposition}
\begin{proof}
	Set $\bals := \cor{\me}$ and let $\eR \in \bals^c$. Then there is an open neighbourhood $\nhd \subset \Rn$ of $\eR$ such that
			\begin{equation*}
				|\me|(\nhd \cap \dom) = 0 \,.
			\end{equation*}
			Now let $\eR' \in \nhd$ and $\nhd{}'\subset \Rn$ be an open neighbourhood of $\eR'$. Then
			\begin{equation*}
				|\me|(\nhd \cap \nhd' \cap \dom) \leq |\me|(\nhd \cap \dom) = 0 \,.
			\end{equation*}
			Thus, $\eR' \in \bals^c$. Since $\eR$ was arbitrary, it follows that for every $\eR \in \bals^c$ there exists an open neighbourhood $\nhd \subset \Rn$ of $\eR$ such that $\nhd \subset \bals^c$, whence $\bals^c$ is open and $\bals$ closed.
\end{proof}
On bounded domains, the \tcor{} is non-empty.
\begin{proposition}\label{prop:cor_nempty}
	Let $\dom \subset \Rn$ be bounded, $\al \subset \pos{\dom}$ be an \tal{} containing every relatively open set in $\dom$ and $\me\in \baa,\me \neq 0$. Then $\cor{\me}$ is non-empty and for every $\delta > 0$
	\begin{equation*}
		|\me|\left(\dom \cap \left((\cor{\me})_\delta\right)^c\right) = 0 \,.
	\end{equation*}
\end{proposition}
\begin{proof} Set $\bals := \cor{\me}$.
	Assume $\cor{\me}$ was empty. Then, by compactness of $\overline{\dom}$
	there exists an open covering $\{\nhd_k\}_{k=\mN}^m$ of $\cl{\dom}$ such that for $k=\mN, ..., m$
	\begin{equation*}
		|\me|(\nhd_k \cap \dom ) = 0 \,.
	\end{equation*}
	But then
	\begin{equation*}
		|\me|(\dom) \leq \sum \limits_{k=\mN}^m |\me|(\nhd_k \cap \dom) = 0
	\end{equation*}
	in contradiction to $\me \neq 0$.

	Now, let $\delta > 0$. For every $\eR \in \overline{(\bals_\delta)^c}^{\Rn}$ there is a $0 < \delta_{\eR} < \frac{\delta}{2}$ such that
	\begin{equation*}
		|\me|\left(B\left(\eR, \delta_{\eR}\right) \cap \dom \right) = 0 \,.
	\end{equation*}
	Otherwise, $\eR \in \cor{\me}$.
	Now
	\begin{equation*}
		\left\{B\left(\eR,\delta_{\eR}\right)\right\}_{\eR \in \overline{\left(\bals_\delta\right)^c}^{\Rn}}
	\end{equation*}
	is an open covering of 
	\begin{equation*}
		\overline{\left(\bals_\delta\right)^c}^{\Rn}  \cap \dom
	\end{equation*}
	Since $\dom$ is relatively compact in $\Rn$, there exists a finite open sub-covering
	\begin{equation*}
		\{B\left(\eR_l,\delta_{\eR_l}\right) \}_{l= \mN}^m 
	\end{equation*}
	of 
	\begin{equation*}
		(\bals_\delta)^c \cap \dom \,.
	\end{equation*}
	Hence
	\begin{equation*}
		|\me|\left(\left(\bals_\delta\right)^c \cap \dom \right) \leq \sum \limits_{l=\mN}^m |\me|\left(B\left(\eR_l, \delta_{\eR_l}\right) \cap \dom\right) = 0 \,.
	\end{equation*}
\end{proof}
\begin{remark}
	If $\dom$ is unbounded, the statement of the preceding proposition need
	not be true. The \tme{}s in Example 10.4.1 in \cite[p.
	245]{rao_theory_1983} can be shown to have empty \tcor{}, since they
	concentrate near infinity.
\end{remark}
The \tcor{} itself does not give all information on the way in which a \tpfa{}
\tme{} concentrates. Hence, the sequences from Proposition
\ref{prop:pfa_iff_supseq}
are investigated further.
\begin{definition}
	Let $\dom \subset \Rn$, $\sal \subset \pos{\dom}$ be a \tsal{}, $\sme
	\in \cas$, $\sme \geq 0$ and $\me_p \in \basws$ be \tpfa{}. Then every $\als\in \sal$ such that
	\begin{equation*}
		|\me_p|(\als^c) = 0
	\end{equation*}
	is called \textbf{\taur}\index{aura@\taur} of $\me_p$. 

	Any decreasing sequence $\{\als_k\}_{k\in \N} \subset \sal$ of auras
	for $\me_p$ such that 
	\begin{equation*}
		\sme(\als_k) \xrightarrow{k\to \infty} 0
	\end{equation*}
	is called \textbf{\tsupseq}\index{aura seq@\tsupseq{}}.
\end{definition}
Now, it is shown that any \tsupseq{} can be restricted to neighbourhoods of the \tcor{}.
\begin{proposition}\label{prop:closed_supseq}
	Let $\dom \subset \Rn$ be bounded and $\sal \subset \pos{\dom}$ be a
	\tsal{} containing every relatively open set in $\dom$. Furthermore,
	let $\sme \in \cas$ with $\sme \geq 0$ and $\me_p \in \basws$ be \tpfa{}. Then for every \tsupseq{} $\{\als_k\}_{k\in \N} \subset \sal$ of $\me_p$ the sequence
	\begin{equation*}
		\{\als'_k\}_{k \in \N} := \left\{\als_k \cap
		\dnhd{(\cor{\me_p})}{\frac{1}{k}} \right\} \subset \sal
	\end{equation*}
	is an \tsupseq{} of $\me_p$ with 
	\begin{equation*}
		\cor{\me_p} = \bigcap \limits_{k \in \N} \overline{\als'_k}^{\Rn} \,.
	\end{equation*}
\end{proposition}
\begin{proof}
	Let $\ferm := \cor{\me_p}$.
	Note that $|\me_p|$ is \tpfa{} and let $\{\als_k\}_{k\in \N} \subset \sal$ be any \tsupseq{} of $\me_p$. Then for every $k \in \N$, $\eR \in \left(\overline{\als_k}^{\Rn}\right)^c$ and any open neighbourhood $\nhd \subset \left(\overline{\als_k}^{\Rn}\right)^c$ of $\eR$
			\begin{equation*}
				|\me_p|(\nhd \cap \dom) \leq |\me_p|\left( \left(\overline{\als_k}^{\Rn}\right)^c \cap \dom\right) \leq |\me_p|(\als_k^c \cap \dom) = 0 \,.
			\end{equation*}
			Hence
			\begin{equation*}
				\ferm \subset \overline{\als_k}^{\Rn} \text{ for every } k \in \N\,.
			\end{equation*}
			Thus,
			\begin{equation*}
				\ferm \subset \bigcap \limits_{k \in \N} \overline{\als_k}^{\Rn} \,.
			\end{equation*}
			For $k\in \N$ set
		\begin{equation*}
			\als'_k := \als_k \cap \dnhd{\ferm}{\frac{1}{k}} \,.
		\end{equation*}
	Then for every $k \in \N$
	\begin{equation*}
		|\me_p|\left(\als_k'^c\right) \leq |\me_p|(\als_k^c) +
		|\me_p|\left(\left(\dnhd{\ferm}{\frac{1}{k}}\right)^c\cap \dom \right)= 0 \,,
	\end{equation*}
	by Proposition \ref{prop:cor_nempty}.

	Furthermore
	\begin{equation*}
		0 \leq \sme(\als'_k) \leq \sme(\als_k) \xrightarrow{k \to \infty} 0 \,.
	\end{equation*}
	Obviously
	\begin{equation*}
		\bigcap \limits_{k \in \N} \overline{\als_k\cap
		\dnhd{\ferm}{\frac{1}{k}}}^{\Rn} \subset \bigcap \limits_{k \in
	\N} \dnhd{\ferm}{\frac{1}{k}} = \ferm \,.
	\end{equation*}
	It remains to show that
	\begin{equation*}
		\ferm \subset \bigcap\limits_{k \in \N} \overline{\als_k \cap
		\dnhd{\ferm}{\frac{1}{k}}}^{\Rn} \,.
	\end{equation*}
	Let $\eR \in \ferm$. Then $\eR \in \overline{\als_k}^{\Rn}$ for every $k$. Hence, for every $k$ there is a sequence $\{\eR^k_l\}_{l \in \N} \subset \als_k$ such that 
	\begin{equation*}
		\eR^k_l \xrightarrow{l \to \infty} \eR \,.
	\end{equation*}
	In particular, there is an $l_0^k \in \N$ such that
	\begin{equation*}
		\|\eR^k_l - \eR \| < \frac{1}{k} \text{ for } l \geq l^k_0 \,.
	\end{equation*}
	Hence, for every $k \in \N$, 
	\begin{equation*}
		\eR \in \overline{\als_k \cap \dnhd{\ferm}{\frac{1}{k}}}^{\Rn} \,.
	\end{equation*}
	Since $\eR \in \ferm$ was arbitrary, this finally implies
	\begin{equation*}
		\ferm \subset \bigcap \limits_{k \in \N} \overline{\als_k \cap
		\dnhd{\ferm}{\frac{1}{k}}}^{\Rn} \,.
	\end{equation*}
\end{proof}
\begin{figure}[H]
	\centering
	\begin{tikzpicture}[rotate=30,scale=1.4]
		\draw (0,0) -> (0,1);
		\draw (0.15,0.5) node {\ferm};
		\draw[gray] (0,0.5) circle [x radius=1.5cm, y radius=1.9cm];
		\draw[gray] (0.0,2.2) node {$A_1$};
		\draw[gray] (0,0.5) circle [x radius=1.0cm, y radius=1.45cm];
		\draw[gray] (0.0,1.8) node {$A_2$};
		\draw[gray] (0,0.5) circle [x radius=0.7cm, y radius=1.1cm];
		\draw[gray] (0.0,1.45) node {...};
		\draw[gray] (0,0.5) circle [x radius=0.5cm, y radius=0.8cm];
		\draw[gray] (0,0.5) circle [x radius=0.3cm, y radius=0.6cm];
	\end{tikzpicture}
	\caption{An \tsupseq{} $\{A_k\}_{k \in \N}$ of a \tdd{1} \tme{} with \tcor{}
	$\ferm= \bigcap \limits_{k \in \N} A_k$}\label{fig:supseq}
\end{figure}

The following lemma identifies a big class of \tpfa{} \tme{}s. In particular,
if the \tcor{} of a \tme{} is a Lebesgue \tnulset{}, the \tme{} is necessarily
\tpfa{}.
\begin{proposition}\label{prop:corenul_pfa}
	Let $\dom \in \bor{\Rn}$ and $\me \in \bawl{\dom}$.

	If $\cor{\me} \cap \dom$ is a $\lem$-\tnulset{} then $\me$ is \tpfa{}.
\end{proposition}
\begin{proof}
	Let $\bals := \cor{\me}$. Then by the definition of the \tcor{}, for every $\delta > 0$
	\begin{equation*}
		|\me|\left(\dnhd{\bals}{\delta}^c \cap \dom \right) = 0 \,.
	\end{equation*}
	Now let $\bals_k := \dnhd{\bals}{\frac{1}{k}} \cap \dom$ for $k \in \N$ and $\sme \in \bawl{\dom}, \sme \geq 0$ be a \tsme{} such that
	\begin{equation*}
		0 \leq \sme \leq |\me| \,.
	\end{equation*}
	Then for every $k \in \N$
	\begin{equation*}
		0 \leq \sme((\bals_k)^c) \leq |\me|((\bals_k)^c) = 0 \,.
	\end{equation*}
	On the other hand, since $\cor{\me}\cap \dom$ is a $\lem$-\tnulset{}, 
	\begin{equation*}
		\sme(\dom\cap \bals) = 0 \,.
	\end{equation*}
	Hence
	\begin{equation*}
		\sme(\dom) = \sme(\dom \cap \bals) + \sme\left ( \bigcup\limits_{k \in N} \bals_k^c \right) = \lim \limits_{k \to \infty} \sme\left(\bals_k^c\right) = 0 \,.
	\end{equation*}
	This implies $\sme = 0$. 

	Since $\sme$ was arbitrary, $\me$ is \tpfa{} by Proposition
	\ref{prop:char_orth} and Proposition \ref{prop:yosh_hew_dec}.
\end{proof}
\begin{remark}
	Note that $\cor{\me}\subset \cl{\dom}$. If $\dom \subset \Rn$ is open
	such that $\lem(\bd{\dom}) >0$, then there is $\me \in \bawl{\dom}$
	such that $\cor{\me} = \bd{\dom}$. Hence $\cor{\me}$ is not a
	\tnulset{}, but $\cor{\me} \cap \dom = \emptyset$. Thus, $\me$ is
	necessarily \tpfa{}.
\end{remark}

The following proposition is taken from \cite[p. 70]{rao_theory_1983}.
It shows that there are many degrees of freedom when choosing an extension of
a \tme{} to a larger class of sets. Since all \tpfa{} \tme{}s used below are
constructed using an extension argument, they are in general not unique.
\begin{proposition}
	Let $\dom \subset \Rn$ and $\al \subset \pos{\dom}$ be an \tal{} on $\dom$. Let $\me \in \baa, \me\geq 0$. Let $\als \in \pos{\dom}\setminus \al$ and $\al'\subset\pos{\dom}$ the smallest \tal{} such that $\al,\{\als\} \subset \al'$.
	Then for any $\const \in [0,\infty)$ such that
	\begin{equation*}
		\sup \{\me(\als') \setpipe \als' \in \al, \als' \subset \als\}
		\leq \const \leq \inf\{\me(\als') \setpipe \als' \in \al, \als \subset \als'\}
	\end{equation*}
	there exists an extension $\me' \in \baA{\al'}, \me'\geq 0$ of $\me$ to all of $\al'$ such that
	\begin{equation*}
		\me'(\als) = \const \,.
	\end{equation*}
\end{proposition}

\section{Integration Theory and $\dual{\libld{\dom}}$}
Now, integration with respect to \tme{} which are not necessarily
$\sigma$-additive is outlined.
Measurability of functions is not defined through the regularity of preimages but by
approximability by \tsimf{} functions in measure. In this definition, the
\tme{} is needed on possibly non-measurable sets. Hence, an \toume{}
has to be used. This \toume{} is defined as in the case of \tsme{}s (cf.
\cite[p. 86]{rao_theory_1983}, \cite[p. 42]{halmos_measure_1974}).
\begin{definition}
	Let $\dom \subset \Rn$ and $\al \subset \pos{\dom}$ be an \tal{}. For
	$\me \in \baa$, $\me \geq 0$ the \textbf{\toume} of $\me$ is defined
	for $\bals \in \pos{\dom}$ by
	\begin{equation*}
		\oume{\me}(\bals) := \inf\limits_{\substack{\als\in \al,\\\bals
		\subset \als}}\me(\als) \,.
		\nomenclature[m]{$\oume{\me}$}{\toume{} for $\me$}
	\end{equation*}
\end{definition}

Now, convergence in measure can be defined. 
The definition is taken from \cite[p. 92]{rao_theory_1983} (cf.
\cite[p. 91]{halmos_measure_1974}).
\begin{definition}
	Let $\dom \subset \Rn$ and $\al \subset \pos{\dom}$ be an \tal{} and
	$\me: \al \to \R$ be a \tme{}. A sequence $\{\fun_k\}_{k \in \N}$ of
	functions $\fun_k : \dom \to \R$ is said to
	converge \textbf{\tconvim{}}\index{convergence \tconvim{}} to a function $\fun : \dom \to \R$ if for every $\eps > 0$
	\begin{equation*}
		\lim \limits_{k \to \infty} \oume{\tova{\me}} \{\eR \in \dom \setpipe |\fun_k(\eR) - \fun(\eR) | > \eps \} = 0 \,.
	\end{equation*}
	In this case, write
	\begin{equation*}
		\fun_k \convim{\me} \fun \,.
		\nomenclature[i]{$\fun_k \convim{\me} \fun$}{$\fun_k$ converge
		\tconvim{} to $\fun$}
	\end{equation*}
	\end{definition}
Note that the limit in measure is not unique, yet. Therefore, the following
notion of equality almost everywhere is needed.
The definition is taken from \cite[p. 88]{rao_theory_1983}.
\begin{definition}
	Let $\dom \subset \Rn$, $\al \subset \pos{\dom}$ and $\me: \al \to \R$ be a measure.

	Then $\fun: \dom \to \R$ is called \textbf{\tnulfun}\index{null
	function@\tnulfun{}}, if for every $\eps >0$
	\begin{equation*}
		\oume{\tova{\me}} \left(\left\{\eR \in \dom \setpipe |\fun(\eR)|>\eps\right\}\right) = 0 \,.
	\end{equation*}

	Two functions $\fun_1: \dom \to \R$, $\fun_2: \dom \to \R$ are called
	\textbf{equal almost everywhere (\tmae{})}\index{equal \tmae{}} with respect to $\me$, if $\fun_1 - \fun_2$ is a \tnulfun{}.

	In this case, write
	\begin{equation*}
		\fun_1 = \fun_2 \mae{\me}
		\nomenclature[i]{$\fun_1 = \fun_2 \mae{\me}$}{$\fun_1 = \fun_2$ \tmae{}}
	\end{equation*}
\end{definition}
\begin{remark}
	If $\fun: \dom \to \R$ is a \tnulfun{}, then it need not be true that
	\begin{equation}\label{eq:nullfunc}
		\oume{\tova{\me}} \left (\left\{\eR\in \dom \setpipe \fun (\eR) \neq 0 \right\}\right ) = 0 \,.
	\end{equation}
	Take e.g. the \tdme{} $\me$ introduced in Example \ref{ex:dzero} and
	$\fun(\eR) := |\eR|$. Then $\fun$ is a \tnulfun{} but
	\begin{equation*}
		\oume{\tova{\me}}(\{\eR\in \Rn | \fun(\eR)\neq 0\} =
		\me(\ball{0}{1}\setminus\{0\}) = 1 > 0 \,.
	\end{equation*}

	This entails that the notion of equality almost everywhere that was
	defined above does not imply the existence of a \tnulset{} such that
	$\fun_1 = \fun_2$ outside of that set. Take e.g. the \tdme{} introduced
	in Example \ref{ex:dzero}, $\fun_1(\eR) := |\eR|$ and
	$\fun_2(\eR):=2\fun_1(\eR)$.

	On the other hand, if $\me$ is a \tsme{} and $\al$ a \tsal{}, then Equation \eqref{eq:nullfunc} is equivalent to $\fun$ being a \tnulfun{} (cf. \cite[p. 89]{rao_theory_1983}).

\end{remark}
The limit in measure turns out to be unique in the sense of almost equality.
This is stated in the following proposition taken from \cite[p. 92]{rao_theory_1983}.
\begin{proposition}
	Let $\dom \subset \Rn$, $\al \subset \pos{\dom}$ be an \tal{} and $\me : \al \to \R$ be a measure. 
	Furthermore let $\{\fun_k\}_{k \in \N}$ be a sequence of functions
	$\fun_k : \dom \to \R$ and $\fun, \tilde{\fun} : \dom \to \R$ be functions such that
	\begin{equation*}
		\fun_k \convim{\me} \fun\,.
	\end{equation*}
	Then
	\begin{equation*}
		\fun_k \convim{\me} \tilde{\fun} \iff	\fun = \tilde{\fun} \mae{\me}
	\end{equation*}
\end{proposition}
Now, the notion of measurability is introduced.
The definition is similar to the definition of $T_1$-measurability in \cite[p. 101]{rao_theory_1983}.
\begin{definition}
	Let $\dom \subset \Rn$ and $\al \subset \pos{\dom}$ be an \tal{} and
	$\me : \al \to \R$ be a \tme{}. A function $\fun : \dom \to \R$ is
	called \textbf{\tmefun{}}\index{measurable function@\tmefun{} function}
	if there exists a sequence $\{\simf_k\}_{k \in \N}$ of \tsimf{}
	functions $\simf_k : \dom \to \R$ such that
	\begin{equation*}
		\simf_k \convim{\me} \fun \,.
	\end{equation*}
\end{definition}
The integral for \tmefun{} functions can now be defined via $\lp^1$-Chauchy
sequences. This is of course well-defined (cf. \cite[p. 102]{rao_theory_1983}).
\begin{definition}
	Let $\dom \subset \Rn$, $\al \subset \pos{\dom}$ be an \tal{} and $\me:
	\al \to \R$ be a \tme{}. A function $\fun : \dom \to \R$ is said to be
	\textbf{\tIfun{}}\index{integrable function@\tIfun{} function} if there
	exists a sequence $\{\simf_k\}_{k\in \N}$ of \tIfun{} \tsimf{}
	functions $\simf_k : \dom \to \R$ such that
	\begin{enumerate}
		\item $\simf_k \convim{\me} \fun$.
		\item $\lim \limits_{k,l \to \infty} \I{\dom}{|\simf_k - \simf_l|}{|\me|} = 0$.
	\end{enumerate}
	In this case, denote
	\begin{equation*}
		\I{\dom}{\fun}{\me} := \lim\limits_{k \to \infty} \I{\dom}{\simf_k}{\me} \,.
		\nomenclature[i]{$\I{\dom}{\fun}{\me}$}{\tI{} of $\fun$ w.r.t. \tme{} $\me$}
	\end{equation*}
		The sequence $\{\simf_k\}_{k\in \N}$ is called
	\textbf{\tdetseq{}}\index{determining seq@\tdetseq{} of an \tIfun{}
	function} for the \tI{} of $\fun$. 
\end{definition}
\begin{remark}
	In particular, \tIfun{} functions are \tmefun{}.
	This notion of \tI{} is also called Daniell-Integral in the literature
	(cf. \cite{rao_theory_1983}).
\end{remark}
The $\lp^p$-spaces are defined in the usual way (cf. \cite[p. 121]{rao_theory_1983}).
\begin{definition}
	Let $\dom \subset \Rn$, $\al \subset \pos{\dom}$ be an \tal{}, $\me: \al \to \R$ be
	a measure and $p\in [1,\infty)$. Then the set of all \tmefun{} functions $\fun : \dom \to \R$ such that $|f|^p$ is $|\me|$-\tIfun{} is denoted by
	\begin{equation*}
		\Lpam{p}\,.
		\nomenclature[f]{$\Lpam{p}$}{$p$-\tIfun{} functions w.r.t to \tal{} $\al$ and \tme{} $\me$}
	\end{equation*}
	If $\al = \bor{\dom}$, write
	\begin{equation*}
		\Lpbm{p}\,.
		\nomenclature[f]{$\Lpbm{p}$}{$\Lpam{p}$ with $\al = \bor{\dom}$}
	\end{equation*}
	For $\fun_1, \fun_2 \in \Lpam{p}$
	\begin{equation*}
		\fun_1 = \fun_2 \mae{\me} 
	\end{equation*}
	defines an equivalence relation.
	The set of all equivalence classes of this relation is denoted by
	\begin{equation*}
		\lpam{p}\,.
		\nomenclature[f]{$\lpam{p}$}{Equivalence classes in $\Lpam{p}$}
	\end{equation*}
	If $\al = \bor{\dom}$, write
	\begin{equation*}
		\lpbm{p}\,.
		\nomenclature[f]{$\lpbm{p}$}{$\lpam{p}$ with $\al = \bor{\dom}$}
	\end{equation*}
\end{definition}
\begin{definition}
	Let $\dom \subset \Rn$, $\al \subset \pos{\dom}$ be an \tal{} and $\me: \al \to \R$ a \tme{}. Then for every $p \in [1,\infty)$ and $\fun \in \Lpam{p}$ write
	\begin{equation*}
		\norm{p}{\fun} := \left(\I{\dom}{|\fun|^p}{|\mu|}\right)^{\frac{1}{p}} \,.
		\nomenclature[n]{$\norm{p}{\fun}$}{$L^p$-norm of $\fun$}
	\end{equation*}
	Furthermore, for \tmefun{} $\fun : \dom \to \R$ define
	\begin{equation*}
		\essup{}{\fun} := \inf \left\{ \K \in \R \setpipe |\me|^*\left(\{\eR\in \dom | \fun(\eR) > \K\}\right) = 0\right\}
	\end{equation*}
	and
	\begin{equation*}
		\normi{\fun} := \essup{}{|\fun|} \,.
	\end{equation*}
	The set of all \tmefun{} functions $\fun : \dom \to \R$ such that
	\begin{equation*}
		\normi{\fun} < \infty
	\end{equation*}
	is denoted by
	\begin{equation*}
		\Liam \,.
	\end{equation*}
	As in the case $p\in [1,\infty)$, 
	\begin{equation*}
		\liam
	\end{equation*}
	denotes the set of all equivalence classes in $\Liam$ with respect to
	equality almost everywhere.

	In the case $\al = \bor{\dom}$, only write
	\begin{equation*}
		\Libm \text{ and } \libm \text{ respectively.}
	\end{equation*}
\end{definition}
The integral defined in this way shares many properties of the
Lebesgue-integral. The Hölder and Minkwoski inequality hold true. Furthermore,
dominated convergence is available when using convergence in \tme{} instead of
pointwise convergence (cf. \cite[p. 105ff]{rao_theory_1983}). 

Before proceeding to the characterisation of the dual of $\lp^\infty$, a new
integral symbol is introduced, which gives formulas for traces
and integrals over \tpfa{} \tme{}s a more pleasing
shape.
\begin{definition}
	Let $\dom \subset \Rn$ be bounded and $\ferm \subset \cl{\dom}$ be
	closed. Then for every $\me \in \bawl{\dom}$ such that
	\begin{equation*}
		\cor{\me} \subset \ferm,
	\end{equation*}
	every $\fun \in \lpbm{1}$ and $\delta >0$ write
	\begin{equation*}
		\sI{\ferm}{\fun}{\me} :=
		\I{\dnhd{\ferm}{\delta}\cap\dom}{\fun}{\me} \,.
		\nomenclature[i]{$\sI{\ferm}{\fun}{\me}$}{$\I{\dnhd{\ferm}{\delta}}{\fun}{\me}$,
		where $\ferm = \cor{\me}$}
	\end{equation*}
\end{definition}
\begin{remark}
	This notion of integral is well-defined since the definition of
	$\cor{\me}$ yields
	\begin{equation*}
		\tova{\me}((\dnhd{\ferm}{\delta})^c) = 0 
	\end{equation*}
	for any $\delta >0$.
\end{remark}
The following proposition is a specialised version of the proposition from
\cite[p. 139]{rao_theory_1983} (cf. \cite[p. 53]{yosida_finitely_1951}).
\begin{proposition}\label{prop:dual_liss_pre}
	Let $\dom \subset \Rn$, $\sal \subset \pos{\dom}$ be a \tsal{} and $\sme : \sal \to \R$ be a \tsme{}.

	Then for every $\de \in \dual{\left( \liss \right)}$ there exists a unique $\me \in \basws$ such that
	\begin{equation*}
		\df{\de}{\fun} = \I{\dom}{\fun}{\me}
	\end{equation*}
	for every $\fun \in \liss$ and
	\begin{equation*}
		\norm{}{\de} = \norm{}{\me} = \tova{\me}(\dom) \,.
	\end{equation*}
	On the other hand, every $\me \in \basws$ defines $\de\in
	\dual{\liss}$.

	Hence, $\dual{\liss}$ and $\basws$ can be identified.
\end{proposition}
Using the decomposition Theorem \ref{thm:dec_pfa_wac} that was proved
earlier, one obtains a more refined characterisation of the dual of
$\liss$. In particular, every element of the dual space is the sum of a \tsme{}
with $\lem$-density and a \tpfa{} \tme{}. In contrast to the literature, this
makes the intuitive idea of the dual of $\lp^\infty$ being $\lp^1$ plus
something which is not \twac{} with respect to Lebesgue measure precise.
\begin{theorem}\label{thm:dual_liss}
	Let $\dom \subset \Rn$ and $\sal \subset \pos{\dom}$ be a \tsal{} and $\sme : \sal \to \R$ be a \tsme{}. Then for every $\dual{u} \in \dual{\liss}$ there exists a unique \tpfa{} $\me_p \in \basws$ and a unique $\dfun \in \lpss{1}$ such that
	\begin{equation*}
		\df{\dual{u}}{\fun} = \I{\dom}{\fun\dfun}{\lem} + \I{\dom}{\fun}{\me_p}
	\end{equation*}
	for every $\fun \in \liss$.
\end{theorem}
\begin{proof}
	Let $\dual{u} \in \dual{\liss}$. Then by Proposition \ref{prop:dual_liss_pre} there exists $\me \in \basws$ such that for all $\fun \in \liss$
	\begin{equation*}
		\df{\dual{u}}{\fun} = \I{\dom}{\fun}{\me} \,.
	\end{equation*}
	Now, by proposition \ref{thm:dec_pfa_wac}, there exist unique $\me_c, \me_p \in \basws$ such that
	\begin{equation*}
		\me = \me_c + \me_p
	\end{equation*}
	and $\me_c$ is a \tsme{} and $\me_p$ is \tpfa{}.
	By the Radon-Nikodym Theorem (Proposition \ref{thm:radon_nik}) there is $\dfun \in \lpss{1}$ such that
	\begin{equation*}
		\me_c(\als) = \I{\als}{\dfun}{\sme}
	\end{equation*}
	for every $\als \in \sal$. Since the integral is obviously linear in $\me$ 
	\begin{equation*}
		\I{\dom}{\fun}{\me} = \I{\dom}{\fun}{\me_c} + \I{\dom}{\fun}{\me_p} = \I{\dom}{\fun \dfun}{\sme} + \I{\dom}{\fun}{\me_p}
	\end{equation*}
	for every $\fun \in \liss$, whence the statement of the proposition follows.
\end{proof}
\begin{remark}
	Note that the $\lp$-space over a \tme{} $\me\geq 0$ is in general not
	complete. Nevertheless, the completion is known to be the set of all
	\tac{} \tme{}s whose $p$-norm is finite, i.e. all bounded \tme{}s
	$\metoo$ with $\metoo \ac \me$ and
	\begin{equation*}
		\lim\limits_{\partition \in \partitions} \sum
		\limits_{\substack{\parte\in \partition\\\me(\parte) \neq 0}} \left |
		\frac{\metoo(\parte)}{\me(\parte)}\right|^p \me(\parte) <
		\infty \,.
	\end{equation*}
	Here, the limit is taken over the directed set $\partitions$ of all partitions
	$\partition$ of $\dom$. See \cite[p. 185ff]{rao_theory_1983} for
	reference. Using the convention $\frac{0}{0} = 0$, this limit is the
	same as the \textit{refinement integral}
	\begin{equation*}
		\RI{\dom}{\left | \frac{\metoo}{\me}\right|^p \me}
	\end{equation*}
	as defined by Kolmogoroff in \cite{Kolmogoroff1930}.
\end{remark}
\section{Density Measures}
This section will present the new class of \tme{}s, called \tdme{}s. These
\tme{}s extend on Example \ref{ex:dzero}.
It turns out that the signed distance function plays an important role.
\begin{definition}
	Let $\dom \subsetneq \Rn$ be non-empty. The function 
	\begin{equation*}
		\distf{\dom} : \Rn \to (-\infty,\infty)
	\end{equation*}
	defined by
	\begin{equation*}
	\distf{\dom}(\eR) := \begin{cases}
		\inf\limits_{\eRR\in \dom} |\eR-\eRR| &\mbox{ if } \eR
		\notin\dom \\
		-\inf\limits_{\eRR \in \dom^c} |\eR-\eRR| &\mbox{ if } \eR \in
		\dom \,.
	\end{cases}
	\nomenclature[s]{$\distf{\dom}$}{signed distance function}
	\end{equation*}
	is called \textbf{signed distance function}.\index{signed distance function}

	For sets $\bals \subset \Rn$ write
	\begin{equation*}
		\distf{\dom}(\bals) := \inf\limits_{\eR\in\bals} \distf{\dom}(\eR) \,.
	\end{equation*}
	Furthermore, neighbourhoods of sets prove useful. Therefore, set
	\begin{equation*}
		\dnhd{\dom}{\delta} := \distf{\dom}^{-1}((-\infty,\delta))
		\nomenclature[s]{$\dnhd{\dom}{\delta}$}{$\{\eR \in \encl{\dom}
		\setpipe \distf{\dom}(\eR) < \delta \}$}
	\end{equation*}
	for $\delta \in \R$.
\end{definition}
\begin{remark}
	Note that $\distf{\dom}$ is \tLcont{}, since it is the sum of two
	\tLcont{} functions. If $\dom\subset \Rn$ is bounded, then by
	\cite[p. 2788]{kraft_measure-theoretic_2016}
	\begin{equation*}
		\ham^{n-1}(\bd{(\dnhd{\dom}{\delta})}) < \infty
	\end{equation*}
	for every $\delta \in \Ran{\distf{\bd{\dom}}}$, the range of
	$\distf{\bd{\dom}}$. Note that there exist $\dom\subset \Rn$ having finite
	perimeter with
	\begin{equation*}
		\lim \limits_{\delta\downarrow 0}
		\ham^{n-1}(\bd{(\dnhd{\dom}{\delta})}) = \infty \,.
	\end{equation*}
	See Kraft \cite[p. 2781]{kraft_measure-theoretic_2016} for reference.
\end{remark}
Now, \tdme{}s can be defined. The basic definition essentially demands the
measure to be a probability measure whose \tcor{} is a Lebesgue \tnulset{}. By
scaling, any bounded positive \tme{} whose support has no volume can be seen as
a \tdme{}.
\begin{definition}
	Let $\dom\in\bor{\Rn}$,	$\ferm \subset \cl{\dom}$ be closed and $\lem(\ferm\cap
	\dom) = 0$. A \tme{} $\me \in \bawl{\dom}$ is called a
	\textbf{\tdme{}}\index{density measure@\tdme{}}
	for $\ferm$, if $\me \geq 0$ and
	for all $\delta > 0$
	\begin{equation*}
		\me(\dnhd{\ferm}{\delta}\cap \dom) = \me(\dom) = 1\,.
	\end{equation*}
	The set of all density measures for $\ferm$ is denoted by
	\begin{equation*}
		\Dme{\ferm} \,.	
		\nomenclature[m]{$\Dme{\ferm}$}{\tdme{}s for $\ferm$}
	\end{equation*}
\end{definition}
\begin{remark}
	If $\lem(\dom \cap \dnhd{\ferm}{\delta}) = 0$ for some $\delta >0$ or
	$\ferm = \emptyset$, then
	\begin{equation*}
		\Dme{\ferm} = \emptyset \,.
	\end{equation*}
\end{remark}
The following proposition shows that \tdme{}s indeed have \tcor{} on $\ferm$
and that they are \tpfa{}.
\begin{proposition}\label{prop:cor_dme}
	Let $\dom \in \bor{\Rn}$ and $\ferm
	\subset \cl{\dom}$ be closed with $\lem(\ferm \cap \dom) =
	0$. Then for every $\me \in \Dme{\ferm}$
	\begin{equation*}
		\cor{\me} \subset \ferm 
	\end{equation*}
	and $\me$ is \tpfa{}.
\end{proposition}
\begin{proof}
	Let $\eR \in \Rn\setminus \ferm$. Let
	\begin{equation*}
		\delta := \frac{1}{2} \distf{\ferm}(\eR) \,.
	\end{equation*}
	Then for every $0 < \tilde{\delta} <\delta$ 
	\begin{equation*}
		\me(\ball{\eR}{\tilde{\delta}}) \leq \me(\dom\setminus
		\dnhd{\ferm}{\delta}) = 0 \,.
	\end{equation*}
	Hence
	\begin{equation*}
		\eR \notin \cor{\me} \,,
	\end{equation*}
	and thus
	\begin{equation*}
		\cor{\me} \subset \ferm \,.
	\end{equation*}
	Finally
	\begin{equation*}
		\lem(\cor{\me} \cap \dom) \leq \lem(\ferm \cap \dom) = 0 \,.
	\end{equation*}
	By Proposition \ref{prop:corenul_pfa}, $\me$ is \tpfa{}.
\end{proof}
Density \tme{}s can be characterised in a way that justifies their name. In
essence, they are densities of other \tme{}s on their \tcor{}.
\begin{proposition}\label{prop:char_dme}
	Let $\dom \in\bor{\Rn}$ and $\ferm \subset \cl{\dom}$ be closed with $\lem(\ferm \cap \dom)
	=0$. A \tme{} $\me \in \bawl{\dom}$ is a \tdme{} for $\ferm$ if
	and only if there exists a \tme{} $\metoo \in
	\bawl{\dom}$ with $\metoo \geq 0$ satisfying
	\begin{equation*}
		\metoo(\dnhd{\ferm}{\delta}\cap \dom) > 0 \text{ for all } \delta > 0 \,,
	\end{equation*}	
	such that for every $\fun \in \libld{\dom}$
	\begin{equation*}
		\I{\dom}{\fun}{\me} \leq \limsup \limits_{\delta
		\downarrow 0} \mI{\dnhd{\ferm}{\delta}\cap \dom}{\fun}{\metoo} \,.
	\end{equation*}
	Then for every $\fun \in \libld{\dom}$
	\begin{equation}\label{eq:dme_limit}
		\sI{\ferm}{\fun}{\me} = \lim \limits_{\delta\downarrow 0}
		\mI{\dnhd{\ferm}{\delta}\cap \dom}{\fun}{\metoo}
	\end{equation}
	if this limit exists.
\end{proposition}
\begin{proof}
	Let $\me \in \bawl{\dom}$.

	Assume there exists $\metoo \in \bawl{\dom}$ with $\metoo \geq 0$ satisfying
	\begin{equation*}
		\metoo(\dnhd{\ferm}{\delta}\cap \dom) > 0 \text{ for all } \delta > 0
	\end{equation*}
	such that for $\fun \in \libld{\dom}$
	\begin{equation*}
		\I{\dom}{\fun}{\me} \leq \limsup \limits_{\delta \downarrow 0}
		\mI{\dnhd{\ferm}{\delta}\cap\dom}{\fun}{\metoo}	\,.
	\end{equation*}
	Note that since
	\begin{equation*}
		\I{\dom}{-\fun}{\me} \leq \limsup \limits_{\delta \downarrow 0}
		\mI{\dnhd{\ferm}{\delta}\cap \dom}{-\fun}{\metoo}	
	\end{equation*} for $\fun \in \libld{\dom}$,
	\begin{equation*}
		\liminf \limits_{\delta \downarrow 0}
		\mI{\dnhd{\ferm}{\delta}\cap \dom}{\fun}{\metoo} \leq
		\I{\dom}{\fun}{\me} \,.
	\end{equation*}
	Then for $\delta > 0$
	\begin{equation*}
		\me(\dom) = \me(\dnhd{\ferm}{\delta}\cap \dom) = \lim \limits_{\delta \downarrow 0}
		\frac{\metoo(\dnhd{\ferm}{\delta}\cap
		\dom)}{\metoo(\dnhd{\ferm}{\delta}\cap \dom)} =1 \,.
	\end{equation*}
	Furthermore, for every $\bals \in \bor{\dom}$
	\begin{equation*}
		\me(\bals) \geq \liminf \limits_{\delta \downarrow 0}
		\frac{\metoo(\bals
		\cap\dnhd{\ferm}{\delta})}{\metoo(\dnhd{\ferm}{\delta}\cap \dom)} \geq 0
		\,.
	\end{equation*}
	Thus, $\me$ is a \tdme{} for $\ferm$. Equation \eqref{eq:dme_limit}
	follows with Proposition \ref{prop:cor_dme} and the previous estimates.
	
	Now assume $\me$ to be a \tdme{} for $\ferm$. Set $\metoo = \me$. Note
	that $\metoo(\dnhd{\ferm}{\delta}\cap \dom) >0$ for every $\delta >0$.
	Then for all $\fun \in \libld{\dom}$
	\begin{equation*}
		\I{\dom}{\fun}{\me} = \lim \limits_{\delta \downarrow 0}
		\I{\dnhd{\ferm}{\delta}\cap \dom}{\fun}{\me} \leq \limsup \limits_{\delta
		\downarrow 0} \mI{\dnhd{\ferm}{\delta}\cap \dom}{\fun}{\metoo} \,.
	\end{equation*}
\end{proof}
Now, existence is proved. It turns out that every \tme{} $\metoo\in
\bawl{\dom}$, which does not vanish near $\ferm$, induces a \tdme{}.
\begin{proposition}\label{prop:ex_densme}
	Let $\dom \in\bor{\Rn}$ and $\ferm\subset \cl{\dom}$ be closed with $\lem(\ferm \cap
	\dom) = 0$. Furthermore, let $\metoo \in \bawl{\dom}$ with $\metoo \geq
	0$ be such that for all $\delta > 0$
	\begin{equation*}
		\metoo(\dnhd{\ferm}{\delta}\cap \dom) > 0 \,.
	\end{equation*}
	Then there exists a \tdme{} $\me \in \bawl{\dom}$ such that for
	every $\fun \in \libld{\dom}$
	\begin{equation*}
		\liminf \limits_{\delta \downarrow 0}
		\mI{\dnhd{\ferm}{\delta}\cap \dom}{\fun}{\metoo} \leq \sI{\ferm}{\fun}{\me} \leq \limsup \limits_{\delta
		\downarrow 0} \mI{\dnhd{\ferm}{\delta}\cap \dom}{\fun
		}{\metoo} \,.
	\end{equation*}
\end{proposition}
\begin{remark}
	In particular, if $\lem(\dnhd{\ferm}{\delta} \cap \dom) >0$  for every
	$\delta >0$, then $\Dme{\ferm}\neq \emptyset$. In order to see this,
	note that $\metoo = \reme{\lem}{\dom}$ satisfies the assumptions of the
	preceding proposition.
	\noindent
	Furthermore, every \tdme{} arises in
	this way (cf. Proposition \ref{prop:char_dme}).
\end{remark}
\begin{proof}
	Let $\metoo \in \bawl{\dom}$ be such that for every $\delta > 0$
	\begin{equation*}
		\metoo(\dnhd{\ferm}{\delta} \cap \dom) > 0 \,.
	\end{equation*}
	Then
	\begin{equation*}
		p : \libld{\dom} \to \R : \fun \mapsto \limsup
		\limits_{\delta \downarrow 0 }\mI{\dnhd{\ferm}{\delta} \cap
		\dom}{\fun}{\metoo}
	\end{equation*}
	is a positively homogeneous, subadditive functional. Set $\BS :=
	\libld{\dom}$ and 
	\begin{equation*}
		\subspace{\BS} := \left\{\fun \in \BS \setpipe \lim \limits_{\delta
				\downarrow 0} \mI{\dnhd{\ferm}{\delta} \cap \dom}{\fun}{\metoo} \text{
		exists}\right\} \,.
	\end{equation*}
	Then $\subspace{\BS}$ is a linear subspace of $\BS$ and
	\begin{equation*}
		\de_0 : \subspace{\BS} \to \R : \fun \mapsto \lim
		\limits_{\delta \downarrow 0}\mI{\dnhd{\ferm}{\delta} \cap \dom}{\fun}{\metoo}
	\end{equation*}
	is a continuous linear functional which is bounded by $p$.
	The subadditive version of the Hahn-Banach theorem
	\cite[p. 62]{dunford_linear_1988}
	yields the existence of a linear extension $\de$ of $\de_0$ to all of $\BS$
	which is bounded by $p$. Note that for every $\fun \in
	\libld{\dom}$
	\begin{equation*}
		\df{\de}{\fun} \leq p(\fun) \leq
		\normi{\fun}
	\end{equation*}
	since $\metoo \wac \lem$. Hence, $\de$ is a continuous linear
	functional on $\libld{\dom}$. By Proposition \ref{prop:dual_liss_pre}, there exists $\me \in
	\bawl{\dom}$ such that for every $\fun \in \libld{\dom}$
	\begin{equation*}
		\df{\de}{\fun} = \I{\dom}{\fun}{\me} \,.
	\end{equation*}
	Note that for every $\fun \in \libld{\dom}$
	\begin{equation*}
		\I{\dom}{-\fun}{\me} \leq p(-f) = \limsup\limits_{\delta \downarrow 0}
		\mI{\dnhd{\ferm}{\delta} \cap \dom}{- \fun}{\metoo} 
	\end{equation*}
	which implies
	\begin{equation*}
		\liminf\limits_{\delta \downarrow 0}
		\mI{\dnhd{\ferm}{\delta} \cap \dom}{\fun}{\metoo} \leq \I{\dom}{\fun}{\me}
		\,.
	\end{equation*}
	Now it is easy to see that for every $\bals\in \bor{\dom}$
	\begin{equation*}
		0 \leq \liminf\limits_{\delta \downarrow 0}
		\mI{\dnhd{\ferm}{\delta} \cap \dom}{\ind{\bals}}{\metoo} \leq \me(\bals) \,.
	\end{equation*}
	Hence, $\me \geq 0$. 
	Furthermore,
	\begin{equation*}
		1 = \liminf\limits_{\delta \downarrow 0}
		\mI{\dnhd{\ferm}{\delta} \cap \dom}{\ind{\dom}}{\metoo} \leq
		\me(\dom) \leq \limsup\limits_{\delta
		\downarrow 0}
		\mI{\dnhd{\ferm}{\delta} \cap \dom}{\ind{\dom}}{\metoo} = 1 \,.
	\end{equation*}	
	Finally, let $\tilde{\delta} > 0$. Then 
	\begin{equation*}
		1 = \liminf\limits_{\delta \downarrow 0}
		\mI{\dnhd{\ferm}{\delta} \cap
			\dom}{\ind{\dnhd{\ferm}{\tilde{\delta}}
		\cap \dom}}{\metoo} \leq
		\me(\dnhd{\ferm}{\tilde{\delta}} \cap \dom) \leq \limsup\limits_{\delta
		\downarrow 0}
		\mI{\dnhd{\ferm}{\delta} \cap
			\dom}{\ind{\dnhd{\ferm}{\tilde{\delta}}
		\cap \dom}}{\metoo} = 1 \,.
	\end{equation*}
	Thus, $\me$ is a \tdme{} of $\ferm$.
\end{proof}
\begin{example}
	Let $\dom \subset \R^2$ be a cusped set as in Figure
	\ref{fig:cusp_dens} below and $\ferm = \{\eR\}$, where $\eR \in \R^2$
	is the point at the cusp. Then for every $\delta >0$ 
	\begin{equation*}
		\lem(\dnhd{\ferm}{\delta} \cap \dom) > 0 \,.
	\end{equation*}
	Hence there exists a \tdme{} $\me \in \Dme{\ferm}$ such that for every
	$\fun \in \libld{\dom}$
	\begin{equation*}
		\sI{\ferm}{\fun}{\me} = \lim \limits_{\delta\downarrow
		0}\mI{\dnhd{\ferm}{\delta}\cap \dom}{\fun}{\lem} \,,
	\end{equation*}
	if this limit exists. This example is in essence identical to Example
	\ref{ex:dzero}.
\end{example}
\begin{figure}[H]
	\centering
	\begin{tikzpicture}[scale=2.0]
		\draw (2,0) to[out=180,in=0] (0,1);
		\draw (0,1) to[out=180,in=90] (-1,0);
		\draw (-1,0) to[out=-90,in=180] (0,-1);
		\draw (0,-1) to[out=0,in=180] (2,0);
		\draw (0,0) node {$\dom$};
		\draw (1.95,0.05) to (2.05,-0.05);
		\draw (1.95,-0.05) to (2.05,0.05);
		\draw (2.2,0) node {$\ferm$};
		\draw[color=gray] (2,0) circle [radius = 1.5];
		\draw[color=gray] (0.5,1.0) node {$\dnhd{\ferm}{\frac{3}{2}}$};
		\draw[color=gray] (2,0) circle [radius = 1.0];
		\draw[color=gray] (1.0,0.7) node {$\dnhd{\ferm}{1}$};
		\draw[color=gray] (2,0) circle [radius = 0.4];
		\draw[color=gray] (1.5,0.4) node {$\dnhd{\ferm}{\frac{2}{5}}$};
	\end{tikzpicture}
	\caption{Existence of a \tdme{} at a cusp}\label{fig:cusp_dens}
\end{figure}
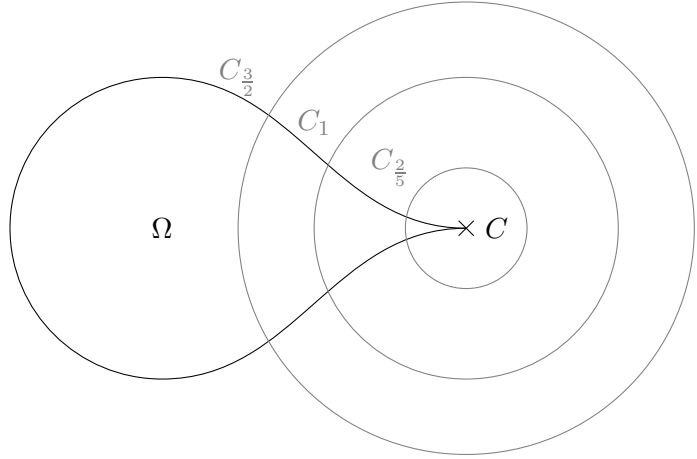
The integral with respect to a \tdme{} can be estimated by the essential
supremum and the essential infimum of the integrand near the \tcor{}.
\begin{proposition}\label{prop:cl_bds_dme}
	Let $\dom\in \bor{\Rn}$ and $\ferm\subset \cl{\dom} $ be closed with $\lem(\ferm
	\cap\dom) = 0$. Furthermore, let $\me
	\in \bawl{\dom}$ be a \tdme{} of $\ferm$. Then for every $\fun
	\in \libld{\dom}$
	\begin{equation*}
		\lim \limits_{\delta \downarrow 0}
		\essinf{\dnhd{\ferm}{\delta} \cap \dom}{\fun} \leq
		\sI{\ferm}{\fun}{\me} \leq \lim \limits_{\delta \downarrow
		0} \essup{\dnhd{\ferm}{\delta} \cap \dom}{\fun}
	\end{equation*}
\end{proposition}
\begin{proof}
	It suffices to prove the right-hand side of the inequality.

	Let $\fun \in \libld{\dom}$. Since $\me \geq 0$, for every $\delta
	> 0$
	\begin{equation*}
		\I{\dom}{\fun}{\me} = \I{\dnhd{\ferm}{\delta} \cap \dom}{\fun}{\me}
		\leq
		\I{\dnhd{\ferm}{\delta} \cap
		\dom}{\essup{\dnhd{\ferm}{\delta} \cap \dom}{\fun}}{\me}
		= 
		\essup{\dnhd{\ferm}{\delta} \cap \dom}{\fun} \,.
	\end{equation*}
	$\essup{\dnhd{\ferm}{\delta} \cap \dom}{\fun}$ is increasing in
	$\delta >0$ and bounded.
	Passing to the limit yields the statement.
\end{proof}
If $\Dme{\ferm} \neq \emptyset$ is ensured, then the
inequalities in the preceding proposition are sharp.
\begin{proposition}\label{prop:op_bds_dme}
	Let $\dom \in \bor{\Rn}$ and $\ferm \subset \cl{\dom}$ be non-empty,
	closed with $\lem(\ferm \cap \dom) = 0$ such that for every $\delta >0$
	\begin{equation*}
		\lem(\dnhd{\ferm}{\delta} \cap \dom) >0 \,.
	\end{equation*}
	Furthermore, let $\fun \in \libld{\dom}$. Then
	\begin{equation*}
		\sup \limits_{\me \in \Dme{\ferm}} \sI{\ferm}{\fun}{\me} =
		\lim \limits_{\delta \downarrow 0}
		\essup{\dnhd{\ferm}{\delta} \cap \dom}{\fun} 
	\end{equation*}
	and
	\begin{equation*}
		\inf \limits_{\me \in \Dme{\ferm}} \sI{\ferm}{\fun}{\me} =
		\lim \limits_{\delta \downarrow 0}
		\essinf{\dnhd{\ferm}{\delta} \cap \dom}{\fun} \,.
	\end{equation*}
\end{proposition}
\begin{proof}
	Let $\fun \in \libld{\dom}$ and $\eps > 0$. Set 
	\begin{equation*}
		M_\eps := \{\eR \in\dom \setpipe \fun(\eR) \geq \lim
			\limits_{\delta \downarrow
		0}\essup{\dnhd{\ferm}{\delta} \cap \dom}{\fun} - \eps\}
	\end{equation*}
	and
	\begin{equation*}
		\metoo_\eps := \reme{\lem}{M_\eps} \,.
	\end{equation*}
	Then $\metoo_\eps \in \bawl{\dom}$ is positive and such that for
	every $\delta >0$
	\begin{equation*}
		\metoo_\eps(\dnhd{\ferm}{\delta} \cap \dom) > 0 \,.
	\end{equation*}
	Hence by Proposition \ref{prop:dual_liss_pre} , there exists a \tdme{}
	$\me_\eps \in \bawl{\dom}$ of $\dom$
	such that
	\begin{equation*}
		\I{\dom}{\fun}{\me_\eps} \geq \liminf \limits_{\delta
		\downarrow 0} \mI{\dom}{\fun}{\metoo_\eps} \geq \lim
		\limits_{\delta\downarrow 0}
		\essup{\dnhd{\ferm}{\delta} \cap \dom}{\fun} - \eps \,.
	\end{equation*}
	Hence
	\begin{equation*}
		\sup \limits_{\me \in \Dme{\ferm}} \I{\dom}{\fun}{\me}
		\geq \sup \limits_{\eps > 0} \I{\dom}{\fun}{\me_\eps} \geq
		\lim \limits_{\delta\downarrow 0}
		\essup{\dnhd{\ferm}{\delta} \cap \dom}{\fun} \,.
	\end{equation*}
	On the other hand, Proposition \ref{prop:cl_bds_dme} yields
	\begin{equation*}
		\sup \limits_{\me \in \Dme{\ferm}} \I{\dom}{\fun}{\me}
		\leq \lim \limits_{\delta\downarrow 0}
		\essup{\dnhd{\ferm}{\delta} \cap \dom}{\fun} \,.
	\end{equation*}
	The statement for $\essinf{}{}$ follows analogously.
\end{proof}
The set of all \tdme{}s is a weak* compact convex set, as the following
proposition shows.
\begin{proposition}
	Let $\dom\in \bor{\Rn}$,
	$\ferm\subset \cl{\dom}$ be non-empty, closed such that $\lem(\ferm \cap
	\dom) = 0$. Then
	$\Dme{\ferm}$ is a convex weak* compact subset of $\bawl{\dom}$ as the
	dual of $\libld{\dom}$.
\end{proposition}
\begin{proof}
	W.l.o.g. $\Dme{\ferm} \neq \emptyset$.

	Let $\me_1,\me_2\in \Dme{\ferm}$ and $\coef_1, \coef_2 \in [0,1]$ such
	that $\coef_1 + \coef_2 = 1$. Then for every $\delta > 0$
	\begin{equation*}
		\coef_1\me_1(\dnhd{\ferm}{\delta} \cap \dom) +
		\coef_2\me_2(\dnhd{\ferm}{\delta} \cap \dom) =
		\coef_1\me_1(\dom) +
		\coef_2\me_2(\dom) = \coef_1 + \coef_2 = 1 \,.
	\end{equation*}
	and 
	\begin{equation*}
		\coef_1\me_1 + \coef_2 \me_2 \geq \coef_1 \me_1 \geq 0 \,.
	\end{equation*}
	Hence, $\Dme{\ferm}$ is a convex set.

	For $\me \in \Dme{\ferm}$
	\begin{equation*}
		\norm{}{\me} = \tova{\me}(\dom) = \me(\dom) = 1 \,.
	\end{equation*}
	Hence, $\Dme{\ferm}$ is a bounded set.

	Now let $\metoo \in \bawl{\dom} \setminus \Dme{\ferm}$. Then
	either $\metoo(\dom) \neq 1$ or there is a $\delta > 0$ such
	that $\metoo(\dnhd{\ferm}{\delta} \cap \dom) \neq 1$ or there is $\bals \in
	\bor{\dom}$ such that $\metoo(\bals) < 0$.

	Consider the first case. Set $\eps := \frac{1}{2}|\metoo(\dom)
	-1|$. Then
	\begin{equation*}
		\nhd(\metoo) := \{\me \in \bawl{\dom} \setpipe 
		|\me(\dom) - \metoo(\dom)| < \eps\}
	\end{equation*}
	is a weak* open set such that
	\begin{equation*}
		\nhd(\metoo) \cap \Dme{\ferm} = \emptyset \,.
	\end{equation*}
	
	In the second case set $\eps :=\frac{1}{2} |\metoo(\dnhd{\ferm}{\delta}
	\cap \dom)-1|$ and
	\begin{equation*}
		\nhd(\metoo) := \{\me\in\bawl{\dom} \setpipe
			|\me(\dnhd{\ferm}{\delta} \cap \dom) -
			\metoo(\dnhd{\ferm}{\delta} \cap \dom)|
			< \eps\}
	\end{equation*}
	is a weak* open set and 
	\begin{equation*}
		\nhd(\metoo) \cap \Dme{\ferm} = \emptyset \,.
	\end{equation*}

	In the third and final case set $\eps :=\frac{1}{2} |\metoo(\bals)|$ and
	\begin{equation*}
		\nhd(\metoo) := \{\me \in\bawl{\dom} \setpipe |\me(\bals) -
		\metoo(\bals)| < \eps \} \,.
	\end{equation*}
	Also in this case
	\begin{equation*}
		\nhd(\metoo) \cap \Dme{\ferm} = \emptyset \,.
	\end{equation*}
	Since $\metoo$ was arbitrary, the complement of $\Dme{\ferm}$ is weak*
	open and thus, $\Dme{\ferm}$ is weak* closed. The statement of the
	proposition follows by the Banach-Alaoglu/Alaoglu-Bourbaki Theorem (cf.
	\cite[p. 777]{zeidler_nonlinear_1986}).
\end{proof}
Now, the action of $\Dme{\ferm}$ on a fixed essentially bounded function can be
characterised.
\begin{corollary}
	Let $\dom\in \bor{\Rn}$,
	$\ferm \subset \cl{\dom}$ be non-empty, closed such that $\lem(\ferm
	\cap\dom) = 0$ and for every $\delta >0$
	\begin{equation*}
		\lem(\dnhd{\ferm}{\delta} \cap \dom) > 0 \,.
	\end{equation*}
	Furthermore, let $\fun \in \libld{\dom}$. 
	
	Then 
	\begin{equation*}
		\df{\Dme{\ferm}}{\fun} = \left[\lim \limits_{\delta \downarrow 0}
			\essinf{\dnhd{\ferm}{\delta} \cap \dom}{\fun}, \lim
			\limits_{\delta \downarrow 0}
		\essup{\dnhd{\ferm}{\delta} \cap \dom}{\fun}\right] \,.
	\end{equation*}
\end{corollary}
\begin{proof}
	Since $\Dme{\ferm}$ is a weak* compact convex subset of
	$\bawl{\dom}$
	\begin{equation*}
		\df{\Dme{\ferm}}{\fun}	
	\end{equation*}
	is a convex compact subset of $\R$. In order to see this, note that
	\begin{equation*}
		\fun \in \dual{\bawl{\dom}} \,.
	\end{equation*}
	Since continuous images of compact sets are again compact,
	\begin{equation*}
		\df{\Dme{\ferm}}{\fun}
	\end{equation*}
	is compact. The convexity follows from the convexity of $\Dme{\ferm}$.
	By Proposition \ref{prop:cl_bds_dme} and Proposition
	\ref{prop:op_bds_dme} 
	\begin{align*}
		\left( \lim
	\limits_{\delta \downarrow 0}
	\essinf{\dnhd{\ferm}{\delta} \cap \dom}{\fun},\lim
	\limits_{\delta \downarrow 0}
\essup{\dnhd{\ferm}{\delta} \cap \dom}{\fun}\right) & \subset
\df{\Dme{\ferm}}{\fun} \\
	&\subset \left [\lim
	\limits_{\delta \downarrow 0}
	\essinf{\dnhd{\ferm}{\delta} \cap \dom}{\fun},\lim
	\limits_{\delta \downarrow 0}
\essup{\dnhd{\ferm}{\delta} \cap \dom}{\fun} \right]
	\end{align*}
	This, together with the fact that $\df{\Dme{\ferm}}{\fun}$ is closed, 
	implies the statement.
\end{proof}
Recall that for a convex set $M$ in a locally convex topological vector space $m\in M$ is an
extremal point if for every $m_1,m_2 \in M$ with $m_1 \neq m_2$ and $\coef_1,\coef_2 \in [0,1]$
with $\coef_1 + \coef_2=1$
\begin{equation*}
	m = \coef_1 m_1 + \coef_2 m_2 \implies \coef_1 = 1- \coef_2 \in \{0,1\} \,.
\end{equation*}
The importance of extremal points follows from the theorem of Krein-Milman (cf.
\cite[p. 154]{eidelman_functional_2004}, \cite[p. 157]{zeidler_nonlinear_1986-1}). In particular, every compact convex set is the closure of the
convex hull of its extremal points. Note that the theorem also implies that the
set of extremal points is non-empty. Hence, the extremal points of $\Dme{\ferm}$
can be regarded as spanning $\Dme{\ferm}$. The following proposition gives a
sufficient and necessary condition for a \tdme{} to be an extremal point.
\begin{proposition}
	Let $\dom \in \bor{\Rn}$,
	$\ferm\subset\cl{\dom}$ be non-empty, closed such that $\lem(\ferm\cap \dom) =
	0$ and $\me \in \Dme{\ferm}$.

	Then $\me$ is an extremal point of $\Dme{\ferm}$ if and only if
	for every $\bals \in \bor{\dom}$ either $\me(\bals) = 0$
	or $\me(\bals^c)=0$.
\end{proposition}
\begin{proof}
	Let $\me \in \Dme{\ferm}$ be such that for every $\bals \in \bor{\dom}$
	either $\me(\bals) = 0$ or $\me(\bals^c) =0$. Assume $\me =
	\coef_1\me_1 + \coef_2\me_2$ for $\me_1,
	\me_2 \in \Dme{\ferm}$ and $\coef_1,\coef_2 \in (0,1)$ such that
	$\coef_1 + \coef_2 = 1$ and $\me_1, \me_2 \neq \me$.
	Then there is $\bals \in \bor{\dom}$ such that
	\begin{equation*}
		\me_1(\bals)\neq \me_2(\bals) \,.
	\end{equation*}
	Suppose $\me(\bals)=0$. Then $\me_1(\bals) = \me_2(\bals) = 0$, a
	contradiction to the assumption.

	Hence $\me(\bals) = 1$ and $\me(\bals^c) = 0$.

	This implies
	\begin{equation*}
		\me_1(\bals^c) = \me_2(\bals^c) = 0 
	\end{equation*}
	and thus
	\begin{equation*}
		\me_1(\bals) = 1 = \me_2(\bals)\,,
	\end{equation*}
	a contradiction to the assumption.
	
	Hence $\me_1 = \me_2 =\me$ and
	$\me$ is an extremal point of $\Dme{\ferm}$.

	Now, assume $\me$ to be an extremal point of $\Dme{\ferm}$ and 
	assume, there exists $\bals \in \bor{\dom}$ such that
	$\me(\bals),\me(\bals^c) > 0$. Set
	\begin{align*}
		\me_1 &:= \frac{1}{\me(\bals)}\reme{\me}{\bals} \\	
		\me_2 &:= \frac{1}{\me(\bals^c)}\reme{\me}{\bals^c} \,.
	\end{align*}
	Then $\me_1$ and $\me_2$ are \tdme{}s and 
	\begin{equation*}
		\me = \me(\bals) \me_1 + \me(\bals^c) \me_2 \,,
	\end{equation*}
	and $\me$ is not an extremal point of $\Dme{\ferm}$ in contradiction to
	the assumption.
\end{proof}
A simple consequence is that the \tcor{} of extremal points contains exactly
one point. This is the same in the case of \tRaM{}, where the Dirac-measures are the
extremal points of the unit ball (cf. \cite[p. 156]{eidelman_functional_2004}). 
\begin{corollary}\label{cor:ex_dme_singl}
	Let $\dom \in \bor{\Rn}$,
	$\ferm \subset \cl{\dom}$ be non-empty, closed, $\lem(\ferm\cap \dom) =
	0$ and $\me \in
	\bawl{\dom}$ be an extremal point of $\Dme{\ferm}$.

	Then $\cor{\me}$ is a singleton.
\end{corollary}
\begin{proof}
	Assume there were $\eR, \eRR \in \cor{\me}$ such that $\eR\neq \eRR$.
	Let $\delta > 0$ be such that $\delta <\frac{1}{2}|\eR-\eRR|$. Then either
	\begin{equation*}
		\me(\ball{\eR}{\delta}) = 0 \text{ or }
		\me(\ball{\eRR}{\delta}^c) = 0 
	\end{equation*}
	in contradiction to $\eR,\eRR \in \cor{\me}$.
\end{proof}
Another obvious corollary gives the values of extremal points on sets $\bals$
whose boundary does not meet the \tcor{} of the extremal point.
\begin{corollary}
	Let $\dom \subset \Rn$,
	$\ferm \subset \cl{\dom}$ be non-empty, closed, $\lem(\ferm\cap \dom) =
	0$ and $\me \in
	\bawl{\dom}$ be an extremal point of $\Dme{\ferm}$ with
	$\cor{\me} = \{\eR\}$ for some $\eR \in \cl{\dom}$. 
	
	Then for every $\bals \in \bor{\dom}$
	\begin{equation*}
		\me(\bals) = \begin{cases}
			1 &\mbox{ if } \eR \in \Int{\bals}, \\
			0 &\mbox{ if } \eR \notin \cl{\bals}\,.
		\end{cases}
	\end{equation*}
\end{corollary}
The question arises, what happens on sets whose boundary meets the \tcor{}.
The following proposition gives a partial answer to this. It states that
extremal points concentrate along one-dimensional directions.
\begin{proposition}
	Let $\dom \in \bor{\Rn}$, $\ferm \subset \cl{\dom}$ be closed with
	$\lem(\ferm \cap \dom) = 0$ and $\me \in \Dme{\ferm}$ be an extremal
	point. Then there exist unique $\eR \in \ferm$ and $\vvec \in \Rn$ with
	$\norm{}{\vvec}=1$ such that for every $\alpha \in \left(0,\frac{\pi}{2}\right)$
	\begin{equation*}
		\me(K(\eR,\vvec,\alpha)\cap \dom) = 1 \,,
	\end{equation*}
	where 
	\begin{equation*}
		K(\eR,\vvec,\alpha) := \{\eRR \in \Rn | \eRR \neq \eR, \sphericalangle
		(\eRR-\eR, \vvec) < \alpha \} \,.
	\end{equation*}
\end{proposition}
\begin{proof}
	By Corollary \ref{cor:ex_dme_singl}, there is a unique $\eR \in \ferm$ such that
	\begin{equation*}
		\cor{\me} = \{\eR\} \,.
	\end{equation*}
	Let $\{\alpha_k\}_{k \in \N} \subset \left(0,\frac{\pi}{2}\right)$ be such that
	\begin{equation*}
		\lim \limits_{k \to \infty} \alpha_k = 0 \,.
	\end{equation*}
	Let $S^n := \bd{\ball{0}{1}}$ and for every $k \in \N$ and $\vvec \in
	S^n$
	\begin{equation*}
		\nhd^k_\vvec := \{\vvec' \in S^n | \sphericalangle(\vvec,\vvec')
		< \alpha_k\} \,.
	\end{equation*}
	Then for each $k \in \N$
	\begin{equation*}
		\left\{\nhd^k_\vvec\right\}_{\vvec \in S^n}
	\end{equation*}
	is an open covering of $S^n$.
	Assume that for every $\vvec \in S^n$
	\begin{equation*}
		\me(K(\eR,\vvec,\alpha_k)\cap \dom) = 0 \,.
	\end{equation*}
	Since $S^n$ is compact, there exists a finite set $M \subset S^n$ such
	that
	\begin{equation*}
		S^n \subset \bigcup \limits_{\vvec \in M} \nhd^k_\vvec \,.
	\end{equation*}
	But then
	\begin{equation*}
		\ball{\eR}{1} \cap \dom \subset \left(\{\eR\} \cup \bigcup\limits_{\vvec \in M}
		K(\eR,\vvec,\alpha_k)\right) \cap \dom \,.
	\end{equation*}
	Hence
	\begin{equation*}
		\me(\dom) = \me(\ball{\eR}{1} \cap \dom) \leq \me(\{\eR\} \cap \dom) + \sum
		\limits_{\vvec \in M} \me\left(K(\eR,\vvec,\alpha_k) \cap
		\dom\right) = 0 \,,
	\end{equation*}
	in contradiction to 
	\begin{equation*}
		\me(\dom) = 1 \,.
	\end{equation*}
	Hence, for every $k \in \N$, there exists $\vvec_k \in S^n$ such that
	\begin{equation*}
		\me(K(\eR,\vvec_k,\alpha_k)\cap \dom) = 1 \,.
	\end{equation*}
	Since $S^n$ is compact, up to a subsequence
	\begin{equation*}
		\vvec_k \xrightarrow{k \to \infty}: \vvec \in S^n \,.
	\end{equation*}
	Now let $\alpha > 0$ and $k_0 \in \N$ be such that for every $k \in
	\N$, $k\geq k_0$
	\begin{equation*}
		\sphericalangle(\vvec_k, \vvec) < \frac{\alpha}{2} \text{ and }
		\alpha_k < \frac{\alpha}{2} \,.
	\end{equation*}
	Then
	\begin{equation*}
		K(\eR,\vvec,\alpha) \supset K(\eR,\vvec_k,\alpha_k) 
	\end{equation*}
	for every $k \geq k_0$ and thus
	\begin{equation*}
		\me(K(\eR,\vvec,\alpha) \cap \dom) \geq
		\me(K(\eR,\vvec_k,\alpha_k) \cap \dom) = 1 \,.
	\end{equation*}
	In order to prove that $\vvec$ is unique, assume there exists $\vvec'
	\in \Rn$, $\vvec' \neq \vvec$ such that the statement of the
	proposition holds. Set 
	\begin{equation*}
		\alpha := \frac{1}{3}\sphericalangle (\vvec, \vvec')
	\end{equation*}
	and note that
	\begin{equation*}
		K(\eR,\vvec,\alpha) \cap K(\eR,\vvec',\alpha) = \emptyset \,.
	\end{equation*}
	But then 
	\begin{equation*}
		\me(\dom \cap (K(\eR,\vvec,\alpha) \cup K(\eR,\vvec',\alpha))) =
		\me(\dom \cap K(\eR,\vvec,\alpha)) + \me(\dom \cap K(\eR,\vvec',\alpha)) = 2
	\end{equation*}
	a contradiction to $\me(\dom) = 1$.
\end{proof}
\begin{remark}
	The proposition above shows that extremal points in $\Dme{\ferm}$
	concentrate around one dimensional directions. Figure \ref{fig:ex_dens_cone}
	illustrates this. Note that it is only necessary for an extremal point
	of $\Dme{\ferm}$ to concentrate in this way. A sufficient condition
	might be that it concentrates on a cusp but this is still an open
	problem.
\end{remark}
\begin{figure}[H]
	\centering
	\begin{tikzpicture}[scale=2.0]
		\draw (-0.5,-0.5) node {$\dom$};
		\draw (-1,-1) rectangle (2,1);
		\draw (0.95,-0.95) to (1.05,-1.05);
		\draw (0.95,-1.05) to (1.05,-0.95);
		\draw (1.2,-1) node {$\ferm$};
		\draw [thick,->] (1,-1) -- node[above] {$\vvec$} (0.2,0);
		\draw[gray] (1,-1) to (0.8,1.3);
		\draw[gray] (1,-1) to (-1.3,-0.7);
		\draw[gray] (0.6,1.2) node {$K_{\alpha_1}$};
		\draw[gray] (1,-1) to (0.3,1.3);
		\draw[gray] (1,-1) to (-1.3,0);
		\draw[gray] (0.1,1.2) node {$K_{\alpha_2}$};
		\draw[gray] (1,-1) to (-0.2,1.3);
		\draw[gray] (1,-1) to (-1.3,0.5);
		\draw[gray] (-0.3,0.4) node {$K_{\alpha_3}$};
	\end{tikzpicture}
	\caption{The cones on which an extremal point of $\Dme{\ferm}$
	is concentrated}\label{fig:ex_dens_cone}
\end{figure}
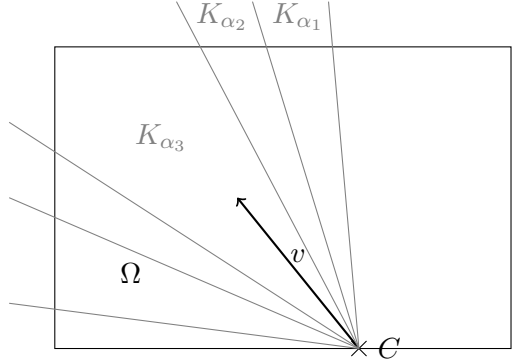

\begin{remark}
	The extremal points of $\Dme{\ferm}$ are called
	\textbf{directionally
	concentrated}\index{directionally concentrated@directionally concentrated \tdme{}} \tdme{}s. 
\end{remark}
Integration with respect to bounded \tdme{}s that was laid out is well-suited for essentially bounded functions $\fun
\in \libld{\dom}$ but in
general it is not suited for unbounded functions. The following example illustrates this.
\begin{example}
	Let $n=1$, $\dom = \ball{0}{1}\subset \R$ and $\ferm :=\{0\}$. Let 
	\begin{equation*}
		\fun(\eR) := \frac{1}{\sqrt{|\eR|}}\left(
		\ind{(-\infty,0)}(\eR) - \ind{[0,\infty)}(\eR) \right ) 
	\end{equation*}
	for $\eR \in \R$.
	Then 
	\begin{equation*}
		\lim \limits_{\delta \downarrow 0} \mI{\ball{0}{\delta}}{\fun}{\lem} = 0 \,.
	\end{equation*}
	Let $\me \in \Dme{\ferm}$ be a \tdme{} of $\ferm$. 
	Then for every $\eps > 0$ and every \tsimf{} $\simf \in \libl$
	\begin{equation}
		\tova{\me}\left(\left\{|\fun - \simf|> \eps\right\}\right) \geq \tova{\me}\left(\left\{|\fun| > \normi{\simf}+\eps\right\}\right) = 1 \,.
	\end{equation}
	Hence there is no sequence of simple function that converge \tconvim{} to $\fun$ and thus $\fun$ is not $\me$-\tIfun{}.
\end{example}
This chapter is closed with some suggestions of further uses for \tdme{}s.
For example, the trace of a function of bounded variation can be computed using \tdme{}s.
\begin{example}
	Let $\dom\subset \Rn$ be bounded with Lipschitz boundary. For
	$\eR \in \bd{\dom}$ let $\me_\eR \in \Dme{\{\eR\}}$ be such that
	\begin{equation*}
		\sI{\{\eR\}}{\fun}{\me_\eR} \leq \limsup \limits_{\delta \downarrow 0} \mI{\ball{\eR}{\delta} \cap
		\dom}{\fun}{\lem}
	\end{equation*}
	for every $\fun \in \libld{\dom}$. Then for every $\bvfun \in
	\BVd{\dom} \cap \libld{\dom}$ and $\ham^{n-1}$-a.e. $\eR \in \bd{\dom}$
	\begin{equation*}
		\bvtr{\dom}(\fun)(\eR) = \sI{\{\eR\}}{\fun}{\me_\eR} \,,
	\end{equation*}
	where $\bvtr{\dom}$ is the usual trace operator for functions of
	bounded variation (cf. \cite[p. 181]{Evans1992}). For fixed $\bvfun \in
	\BVd{\dom}$ it even holds true that for \tmae{} $\eR \in \bd{\dom}$
	\begin{equation*}
		\bvfun \in \lpbmd{1}{\me_\eR}{\dom} \,.
	\end{equation*}
	In order to see this, note that by Evans \cite[p. 181]{Evans1992}
	\begin{equation*}
		\lim \limits_{\delta \downarrow 0} \mI{\ball{\eR}{\delta} \cap
		\dom}{|\bvfun - \bvtr{\dom}(\bvfun)(\eR)|}{\lem} = 0 \,.
	\end{equation*}
	For every $k \in \N$ set
	\begin{equation*}
		\simf_k :=  \bvtr{\dom}(\bvfun)(\eR) \ind{\dom} \,.
	\end{equation*}
	For $\eps > 0$ set
	\begin{equation*}
		\bals_\eps := \{\eRR \in \dom \setpipe |\bvfun(\eRR)-
		\simf_k(\eRR) | \geq \eps \} \,.
	\end{equation*}
	Then
	\begin{align*}
		\me_\eR(\bals_\eps) & \leq \limsup \limits_{\delta\downarrow 0}
		\frac{\lem(\bals_\eps \cap \dom \cap
		\ball{\eR}{\delta})}{\lem(\ball{\eR}{\delta})} \\
		& \leq \limsup \limits_{\delta\downarrow 0} \frac{1}{\eps}
		\mI{\dom\cap
		\ball{\eR}{\delta}}{|\bvfun(\eRR)-\simf_k(\eRR)|}{\lem} \\
		& = \frac{1}{\eps} \lim \limits_{\delta\downarrow 0}
		\mI{\ball{\eR}{\delta}\cap \dom}{|\bvfun -
		\bvtr{\dom}(\bvfun)(\eR)}{\lem} = 0 \,.
	\end{align*}
	Hence, 
	\begin{equation*}
		\simf_k \convim{\me_\eR} \bvfun \,.
	\end{equation*}
	Furthermore, the sequence is constant and thus $\lp^1$-Cauchy. Thus
	$\bvfun \in \lpbmd{1}{\me_\eR}{\dom}$ and
	\begin{equation*}
		\sI{\{\eR\}}{\bvfun}{\me_\eR} = \bvtr{\dom}(\bvfun)(\eR) \,.
	\end{equation*}
	This shows, that even the trace of unbounded functions of bounded
	variation can be expressed this way.
\end{example}
\begin{remark}
	Slightly adapting the technique from the previous example, one can show
	that all unbounded functions are integrable with respect to density
	measures whose \tcor{} is one of the Lebesgue points of the function. This way, traces
	for Sobolev functions and functions of bounded variation can also be
	computed on the interior of the domain.

	For functions of bounded variation, this technique also works at jump
	points, i.e. points where the precise representative is the mean of the
	one-sided traces.
\end{remark}
It is also possible to use \tdme{}s to define a set-valued gradient for \tLcont{}
functions.
\begin{example}
	Let $\ferm = \{\eR\} \subset \Rn$ and $\fun: \Rn \to \R$ be \tLcont{}. Note that by Rademachers Theorem (cf.
	\cite[p. 81]{Evans1992}), $\Deriv{\fun}$ exists almost everywhere and
	is essentially bounded. Set
	\begin{equation*}
		\partial_d\fun(\eR) := \df{\Dme{\{\eR\}}}{\Deriv{\fun}} \,.
	\end{equation*}
	Then $\partial_d\fun(\eR)$ is a weak* compact, convex set which is
	contained in $\ball{0}{L}$, where $L$ is the Lipschitz constant of
	$\fun$.
	In plus, the linearity of the integral implies that
	for every $\fun_1, \fun_2\in \skpdr{1}{\infty}{\Rn}{\R}$
	\begin{equation*}
		\partial_d(\fun_1 + \fun_2)(\eR) \subset \partial_d\fun_1(\eR) +
		\partial_d\fun_2(\eR) \,.
	\end{equation*}
	and
	\begin{equation*}
		\partial_d(\fun_1\fun_2)(\eR) \subset \fun_1(\eR)
		\partial_d(\fun_2)(\eR) + \fun_2(\eR)
		\partial_d(\fun_1)(\eR)\,.
	\end{equation*}
	Note that the definition of $\partial_d$ hints at similarities to a characterisation of Clarkes Generalised Gradient in \cite[p. 63]{clarke_optimization_1984}.
\end{example}
The following proposition states that every \tpfa{} \tme{} induces a \tRaM{} on
its \tcor{}.
\begin{proposition}\label{prop:ram_on_core}
	Let $\dom \in \bor{\Rn}$ be bounded and $\me \in \bawl{\dom}$. 
	Then there exists a Radon measure $\sme$ supported on $\cor{\me}\subset \cl{\dom}$ such that for every $\sfun \in \Cefun{\dom}$
	\begin{equation*}
		\I{\dom}{\sfun}{\me} = \I{\cor{\me}}{\sfun}{\sme} \,.
	\end{equation*}
\end{proposition}
\begin{proof}
	First, note that for every $\sfun \in \Cefun{\dom}$
	\begin{equation*}
		\left |\I{\dom}{\sfun}{\me}\right| \leq \norm{C}{\sfun} \cdot
		\tova{\me}(\dom)
	\end{equation*}
	Furthermore, note that every $\sfun\in \Cefun{\dom}$ can be extended to
	a function $\funex{\sfun} \in \Ccfun{\cl{\dom}}$ and every element of
	$\Ccfun{\cl{\dom}}$ can be restricted to $\dom$ to obtain an element of
	$\Cefun{\dom}$. Hence
	\begin{equation*}
		\de : \Ccfun{\cl{\dom}} \to \R : \sfun \mapsto
		\I{\dom}{\sfun}{\me}
	\end{equation*}
	is a continuous linear operator and by the Riesz Representation Theorem (cf. \cite[p. 106]{federer_geometric_1996}) there is a Radon measure $\sme$ on $\cl{\dom}$ such that for every
	$\sfun \in \Cefun{\dom}$
	\begin{equation*}
		\I{\dom}{\sfun}{\me} = \I{\cl{\dom}}{\sfun}{\sme} \,.
	\end{equation*}
	Now let $\eR \in \cl{\dom}\setminus \cor{\me}$. Then there exists a
	$\delta >0$ such that $\ball{\eR}{\delta} \cap \cor{\me} = \emptyset$.
	
	Then for every $\sfun \in \Ccfun{\ball{\eR}{\delta}\cap \cl{\dom}}$
	\begin{equation*}
		\I{\cl{\dom}}{\sfun}{\sme} = \I{\dom}{\sfun}{\me} = 0 \,.
	\end{equation*}
	Hence
	\begin{equation*}
		\tova{\sme}(\ball{\eR}{\delta}) = 0
	\end{equation*}
	and thus $\eR$ is not in the support of the \tsme{} $\sme$. Since
	$\eR\in \cl{\dom}\setminus\cor{\me}$ was arbitrary, it is proved that
	the support of $\sme$ is indeed a subset of $\cor{\me}$. This proves
	the statement of the proposition.
\end{proof}
\begin{remark}
	In the setting of the proposition above, $\sme$ is said to be a
	\textbf{representation of $\me$} on $\cor{\me}$.
\end{remark}
The next proposition gives a partial inverse to the statement of the
proposition above. In particular, any \tRaM{} can be extended to a \tme{} on
all of its domain.
\begin{proposition}\label{prop:smearing}
	Let $\dom \in \bor{\Rn}$ be bounded and $\ferm \subset \cl{\dom}$ be
	closed such
	that for every $\eR \in \ferm$ and every $\delta > 0$
	\begin{equation*}
		\lem(\ball{\eR}{\delta} \cap \dom) > 0 \,.
	\end{equation*}
	Furthermore, let $\sme$ be a Radon measure on $\ferm$. Then there
	exists $\me \in \bawl{\dom}$ such that for every $\sfun \in
	\Cefun{\dom}$
	\begin{equation*}
		\I{\dom}{\sfun}{\me} = \I{\ferm}{\sfun}{\sme} \,.
	\end{equation*}
	In particular, 
	\begin{equation*}
		\cor{\me} \subset \ferm 
	\end{equation*}
	and
	\begin{equation*}
		\tova{\me}(\dom) = \tova{\sme}(\ferm) \,.
	\end{equation*}
\end{proposition}
\begin{remark}
	The conditions of the statement are satisfied if, for example,  $\ferm \subset \mbd{\dom} \cup \mint{\dom}$.	
\end{remark}
\begin{proof}
	Let $\sfun \in \Cefun{\dom}$. Then 
	\begin{equation*}
		\norm{C}{\refun{\sfun}{\ferm}} \leq \normi{\sfun} \,.
	\end{equation*}
	In order to see this, let $\eps >0$ and $\eR \in
	\ferm$ be such that
	\begin{equation*}
		\left | \sfun(\eR) - \norm{C}{\refun{\sfun}{\ferm}}\right | <
		\frac{\eps}{2} \,.
	\end{equation*}
	Let $\delta >0$ be such that for all $\eRR \in \ball{\eR}{\delta} \cap
	\dom$
	\begin{equation*}
		|\sfun(\eR) - \sfun(\eRR)| < \frac{\eps}{2} \,.
	\end{equation*}
	By assumption 
	\begin{equation*}
		\lem(\ball{\eR}{\delta} \cap \dom) > 0 
	\end{equation*}
	whence
	\begin{equation*}
		\normi{\sfun} \geq |\sfun(\eR)| - \frac{\eps}{2} \geq
		\norm{C}{\refun{\sfun}{\ferm}} - \eps \,.
	\end{equation*}
	Since $\eps > 0$ was arbitrary, the statement follows.

	Set
	\begin{equation*}
		\de_0 : \Cefun{\dom} \subset \libld{\dom} \to \R : \sfun \mapsto
		\I{\ferm}{\sfun}{\sme}
	\end{equation*}
	and note that for every $\sfun \in \Cefun{\dom}$
	\begin{equation*}
		\left |\df{\de_0}{\sfun} \right | \leq
		\norm{C}{\refun{\sfun}{\ferm}} \tova{\sme}(\ferm) \leq
		\normi{\sfun} \tova{\sme}(\ferm) \,.
	\end{equation*}
	By the Hahn-Banach theorem (cf. \cite[p. 63]{dunford_linear_1988})
	there exists a continuous extension $\de$ of $\de_0$ to all of
	$\libld{\dom}$ such that
	\begin{equation*}
		\norm{}{\de} =  \norm{}{\de_0}\,.
	\end{equation*}
	But $\dual{\libld{\dom}} = \bawl{\dom}$ by Proposition
	\ref{prop:dual_liss_pre}.
	Hence, there exists $\me \in \bawl{\dom}$ such that for every $\sfun
	\in \Cefun{\dom}$
	\begin{equation*}
		\I{\dom}{\sfun}{\me} = \I{\ferm}{\sfun}{\sme} \,.
	\end{equation*}
		Let $\sfun\in \Cefun{\dom}$ such that $\normi{\sfun} \leq 1$. Then
	\begin{equation*}
		\norm{C}{\refun{\sfun}{\ferm}} \leq \normi{\sfun} \leq 1 \,.
	\end{equation*}
	Hence,
	\begin{equation*}
		\tova{\me}(\dom) = \norm{}{\de_0} = \sup \limits_{\substack{\sfun\in
		\Cefun{\dom},\\\normi{\sfun}\leq 1}} \I{\dom}{\sfun}{\me} \leq
		\sup \limits_{\substack{\sfun\in
		\Cefun{\dom},\\\norm{C}{\refun{\sfun}{\ferm}} \leq 1}}
		\I{\ferm}{\sfun}{\sme} \leq \tova{\sme}(\ferm) \,.
	\end{equation*}
	Note that for every $\sfun \in \Cefun{\dom}$
	\begin{equation*}
		\max(\min(\sfun, 1),-1) \in \Cefun{\dom}
	\end{equation*}
	and that every $\sfun \in \Ccfun{\ferm}$ can be extended 
	to all of $\cl{\dom}$, preserving the norm (cf.
	\cite[p. 25]{pedersen_analysis_1989}).
	Hence, every $\sfun \in \Ccfun{\ferm}$ can be extended to
	$\funex{\sfun} \in \Ccfun{\cl{\dom}}$ such that
	\begin{equation*}
		\norm{C}{\sfun} = \norm{C}{\funex{\sfun}} \,.
	\end{equation*}
	Thus
	\begin{equation*}
		\tova{\sme}(\ferm) = \sup \limits_{\substack{\sfun \in 
		\Ccfun{\ferm},\\\norm{C}{\sfun} \leq 1}}
		\I{\ferm}{\sfun}{\sme} = \sup
		\limits_{\substack{\funex{\sfun}\in \Ccfun{\cl{\dom}},\\
		\norm{C}{\funex{\sfun}} \leq 1}} \I{\dom}{\funex{\sfun}}{\me} \leq \sup
		\limits_{\substack{\funex{\sfun}\in
		\Ccfun{\cl{\dom}},\\\normi{\funex{\sfun}}\leq 1}}
		\I{\dom}{\funex{\sfun}}{\me} \leq \tova{\me}(\dom) \,.
	\end{equation*}
	Since changing $\sfun$ outside of $\ferm$ does not change the integral,
	$\cor{\me} \subset \ferm$.
	This finishes the proof.
\end{proof}
The \tme{} from the preceding proposition is \tpfa{} if the \tRaM{} is singular
with respect to Lebesgue measure.
\begin{corollary}
	Let $\dom \in \bor{\Rn}$ be bounded and $\ferm \subset \cl{\dom}$ be
	closed such
	that for every $\eR \in \ferm$ and $\delta >0$
	\begin{equation*}
		\lem(\ball{\eR}{\delta} \cap \dom) >0
	\end{equation*}
	and 
	\begin{equation*}
		\lem(\ferm \cap \dom) = 0\,.
	\end{equation*}
	Furthermore, let $\sme$ be a \tRaM{} on $\ferm$.

	Then there exists $\me \in
	\bawl{\dom}$ such that for all $\sfun \in \Ccfun{\dom}$ 
	\begin{equation*}
		\I{\dom}{\sfun}{\me} = \I{\ferm}{\sfun}{\sme}\,.
	\end{equation*}
	Furthermore, 
	\begin{equation*}
		\tova{\me}(\dom) = \tova{\sme}(\ferm)
	\end{equation*}
	and $\me$ is \tpfa{}.
\end{corollary}
\begin{proof}
	The preceding proposition and Proposition \ref{prop:corenul_pfa} yield
	the statement.
\end{proof}
The following example presents another way to construct a density at zero.
\begin{example}
	Let $\dom\in \bor{\Rn}$ be bounded and $\eR \in \cl{\dom}$ such that for every
	$\delta >0$ 
	\begin{equation*}
		\lem(\ball{\eR}{\delta} \cap \dom) > 0 \,.
	\end{equation*}
	Then there exists a \tpfa{} $\me \in \bawl{\dom}$ such that for every
	$\sfun \in \Cefun{\dom}$
	\begin{equation*}
		\I{\dom}{\sfun}{\me} = \sfun(\eR) \,.
	\end{equation*}
\end{example}
The next example shows an extension for $\ham^{n-1}$.
\begin{example}
	Let $\dom \in \bor{\Rn}$ be open, bounded and have smooth boundary. Then $\lem(\bd{\dom}) =
	0$ and $\ferm= \bd{\dom}$ satisfies the assumptions of
	Proposition \ref{prop:smearing}.
	Hence, there exists $\me \in \bawl{\dom}$ such that for all $\sfun \in
	\Cefun{\dom}$
	\begin{equation*}
		\I{\bd{\dom}}{\sfun}{\ham^{n-1}} = \I{\dom}{\sfun}{\me} \,.
	\end{equation*}
\end{example}
The following example shows, that the surface part of a Gauß formula can be
expressed as an integral with respect to a \tpfa{} measure.
\begin{example}
	Let $\dom \in \bor{\Rn}$ be a bounded set with smooth boundary. Then $\ferm = \bd{\dom} \subset
	\cl{\dom}$ is a closed set and for every $k \in \N$ such that $1 \leq
	k \leq n$
	\begin{equation*}
		\bvnu{k} \cdot \reme{\ham^{n-1}}{\bd{\dom}}
	\end{equation*}
	is a \tRaM{} on $\ferm$. By Proposition \ref{prop:smearing} there
	exists $\me_k \in \bawl{\dom}$ such that for every $\sfun \in
	\Cefun{\dom}$
	\begin{equation*}
		\I{\bd{\dom}}{\sfun \cdot \bvnu{k}}{\ham^{n-1}} =
		\I{\dom}{\sfun}{\me_k} = \sI{\bd{\dom}}{\sfun}{\me_k} 
	\end{equation*}
	and 
	\begin{equation*}
		\cor{\me_k} \subset \bd{\dom} \,.
	\end{equation*}
	Hence, there exists $\me \in \bawln{\dom}$ such that for all $\sfun \in
	\ConeN{\cl{\dom}}$
	\begin{equation*}
		\sI{\bd{\dom}}{\sfun}{\me} = \I{\dom}{\sfun}{\me} = \I{\bd{\dom}}{\sfun
		\cdot \normal{}}{\ham^{n-1}} =
			\I{\dom}{\divv{\sfun}}{\lem} \,,
	\end{equation*}
	where the Gauß formula for sets with finite perimeter from Evans \cite[p. 209]{Evans1992} was used.
	Furthermore,
	\begin{equation*}
		\cor{\me} \subset \bd{\dom}
	\end{equation*}
	and $\me$ is \tpfa{} by
	Proposition \ref{prop:corenul_pfa}.
\end{example}
\dobib

\bibliographystyle{plain}
\bibliography{Diss.bib}

\end{document}